\documentclass[11pt]{article}
\usepackage{eurosym}
\usepackage{amssymb}
\usepackage{amsfonts}
\usepackage{amsmath}
\usepackage{graphicx}
\usepackage{hyperref}
\usepackage{subcaption}
\usepackage{tikz}
\usepackage{float}
\usepackage{cite}
\usepackage{float}

\newenvironment{proof}[1][Proof]{\textbf{#1.} }{\ \rule{0.5em}{0.5em}}
\usepackage{pgf,tikz,cancel}
\usepackage{caption}
\usetikzlibrary{arrows}
\usepackage{mathrsfs}
\usepackage{hyperref}

\def\grad{\nabla}
\def\div{\text{div}}
\def\Lap{\Delta}
\def\bnu{\textbf{$\nu$}}
\def\beq{\begin{equation}}
\def\eeq{\end{equation}}
\def\bbm{\begin{bmatrix}}
\def\ebm{\end{bmatrix}}
\def\A{\mathcal{A}}
\def\bA{\mathbb{A}}
\def\e{\epsilon}
\def\l{\lambda}
\def\la{\langle}
\def\ra{\rangle}
\def\Op{{\Omega^+}}
\def\Om{{\Omega^-}}
\def\bup{\mathbf{u^+}}
\def\bum{\mathbf{u^-}}
\def\bU{\mathbf{U}}
\def\R{\mathbb{R}}

\newtheorem{thm}{Theorem}

\newtheorem{lemma}{Lemma}

\newtheorem{fact}{Fact}

\begin{document}

\author{George Avalos \\
Department of Mathematics\\
University of Nebraska-Lincoln \and Paula Egging \\
Department of Mathematics\\
University of Nebraska-Lincoln}
\title{An inf-sup approach to semigroup wellposedness for a compressible flow--incompressible fluid interactive PDE system}
\maketitle

\begin{abstract}
This work presents qualitative and numerical results on a system of partial differential equations (PDEs) which models certain fluid-fluid interaction dynamics. This system models a compressible fluid in a domain $\Op \subset \R^2$, coupled to an incompressible fluid modeled by Stokes flow in domain $\Om \subset \R^2$, with the strong coupling implemented through certain boundary conditions on the shared interface, $\Gamma$. The wellposedness of this system is established by means of constructing for it a semigroup generator representation. This representation is accomplished by eliminating one of the pressure variables via identifying it as the solution of a certain boundary value problem, while the wellposedness is established via a nonstandard usage of the Babuska-Brezzi Theorem. In later sections, we demonstrate how the earlier constructive proof of wellposedness naturally lends itself to a certain finite element method (FEM), by which to numerically approximate solutions of the given coupled PDE system. This FEM is provided with error estimates and associated rates of convergence.

\vskip.3cm \noindent \textbf{Key terms:} Fluid-Fluid Interaction, Strongly Continuous Semigroups, Inf-Sup Approaches, Finite Element Interpolation
\end{abstract}


\section{Introduction}

In this work, we consider a system of partial differential equations (PDEs) which has practical applications in modeling various fluid-fluid interactions which occur in nature, e.g. the interaction between the ocean and atmosphere, two different fronts of air, etc.  A proper modeling of such phenomena can aid in the understanding of current dynamics in the climate, as well as to project models of future climate so as to aid in energy planning (see \cite{Lemarie}). Of course, for fully coupled climate models, the monolithic solution, which treats the entire system as a single unified system, is desired. However, for practical considerations, the full model is always partitioned into submodels, with each component simulated separately. The coupling of these components is achieved through the exchange of boundary condition information in such a way that certain interface conditions are satisfied.
In such models, it is natural to assume the atmosphere (gaseous) can be modeled by a compressible fluid, while the ocean (water) is modeled by an incompressible (divergence-free) fluid (see \cite{Skamarock}). 
The focus of Sections 2-5 is on the abstract mathematical results demonstrating the Hadamard wellposedness of such a fluid-fluid interactions system. Section 6 describes how the proof method employed in previous sections lends itself to the derivation of a numerical method to generate approximate solutions to the given PDE system.  

\section{The PDE and Setting for Wellposedness}
 Towards developing the PDE system under consideration, let $\Op \subset \R^2$ and $\Om \subset \R^2$ be two domains which are open, bounded, disjoint, and convex. Suppose further that they share a common interface $\Gamma$, on which certain (to be specified) boundary transmission conditions will exert a strong coupling between the flows in $\Op$ and $\Om$. 

\begin{center}
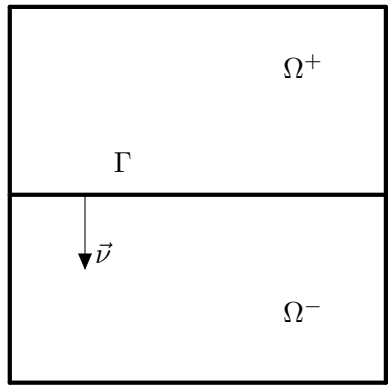
\begin{figure}[h]
\hspace{1.5in}
\begin{tikzpicture}[line cap=round,line join=round,>=triangle 45,x=5.0cm,y=5.0cm]
\clip(-0.25,-0.25) rectangle (1.25,1.25);
\draw(0.,0.) -- (0.,1.) -- (1.,1.) -- (1.,0.) -- cycle;
\draw [line width=1.6pt] (0.,0.)-- (0.,1.);
\draw [line width=1.6pt] (0.,1.)-- (1.,1.);
\draw [line width=1.6pt] (1.,1.)-- (1.,0.);
\draw [line width=1.6pt] (1.,0.)-- (0.,0.);
\draw [line width=1.6pt] (0.,0.5)-- (1.,0.5);
\draw [->] (0.2,0.5) -- (0.2,0.3);
\draw (0.7,0.9) node[anchor=north west] {$\Omega^+$};
\draw (0.7,0.25) node[anchor=north west] {$\Omega^-$};
\draw (0.2,0.4) node[anchor=north west] {$\vec{\nu}$};
\draw (0.25,0.64) node[anchor=north west] {$\Gamma$};
\end{tikzpicture}
\caption{The fluid-fluid geometry.}
 \label{fig:omega}
\end{figure}
\end{center}

With the geometry given, we now introduce the fluid-fluid interaction problem of present concern. The  dependent variables $\bup = [u^+_1(t,x), u^+_2(t,x)]^T$ on $\Op$ and $\bum = [u^-_1(t,x), u_2^-(t,x)]^T$ on $\Om$ will represent the fluid velocity fields, while $p^+(t,x)$ and $p^-(t,x)$  represent the scalar-valued pressure terms in $\Op$ and $\Om$, respectively. With these variables,  consider the boundary value problem given by 

\begin{align}  \label{full_sys}
& \begin{cases}
\mathbf{u}_t^+  + \bU\cdot \grad \bup - \div\, \sigma(\bup) + \grad p^+ = 0 &  \text{ on } \Omega^+ \times (0,T),\\
p_t^+ +  \bU \cdot \grad p^+ + \div(\bup) = 0 & \text{ on } \Omega^+ \times (0,T),\\
\bup = 0 & \text{ on } (\partial \Omega^+ \setminus \Gamma ) \times (0,T),
\end{cases}\\
& \label{um_equns}
\begin{cases}
\mathbf{u}_t^- - \Lap \bum +  \grad p^- = 0 &  \hspace{1.14in} \text{ on } \Omega^- \times (0,T),\\
\div(\bum) = 0 &  \hspace{1.14in} \text{ on } \Omega^- \times (0,T),\\
\bum = 0 &  \hspace{1.14in} \text{ on } (\partial \Omega^- \setminus \Gamma) \times (0,T),
\end{cases}\\
&
\begin{cases} \label{bcs}
\bup = \bum & \hspace{.285in} \text{ on } \Gamma \times (0,T),\\
\sigma(\bup) \vec{\nu} - p^+ \vec{\nu} = \frac{\partial \bum }{\partial \vec{\nu}} - p^- \vec{\nu}&\hspace{.285in}   \text{ on } \Gamma \times (0,T),\\
\bup(t = 0) = \mathbf{u}_0^+; \, \, \, \bum(t = 0) = \mathbf{u}^-_0,&
\end{cases}
\end{align}
where $\vec{\nu}$ is the outward normal vector with respect to $\Omega^+$ (and hence on $\Gamma$, points into $\Omega^-$) and the flow is linearized about a rest state with nonzero ambient background flow given by  $\bU = [U_1, U_2]^T \in \textbf{H}^2(\Omega^+)
 \cap \{\mathbf{ v} \in \textbf{L}^2(\Omega^+): \mathbf{v} \cdot \vec{\nu} = 0 \text{ on } \partial \Omega^+\}$.

Furthermore, the term $\sigma(\mathbf{u})$ is the classic \textit{stress tensor} of the fluid and is given by $$\sigma(\mathbf{u}) = 2 \nu \epsilon(\mathbf{u}) + \tilde\lambda[I_2 \cdot \epsilon(\mathbf{u})]I_2,$$
where the strain tensor $\epsilon$ is given by 
$$\epsilon_{ij}(\mathbf{u}) = \frac12 \left( \frac{\partial u_j}{\partial x_i} + \frac{\partial u_i}{\partial x_j} \right), \, 1 \leq i, j, \leq 2,$$
and $\nu > 0$ and $\tilde\lambda \geq 0$ are Lam\'{e} coefficients of viscosity (see pg. 129 in \cite{Kesavan}). Given this notation, one can verify that $$
\div \, \sigma( \mathbf{u}) = \nu \Lap \mathbf{u} + (\nu + \tilde\lambda) \grad \div(\mathbf{u}).
$$
Throughout, the bold-faced spaces, such as $\mathbf{L}^2(\Op)$, represent spaces of vector-valued functions, while non-bold font, such as $L^2(\Op)$, will represent a space of scalar-valued functions.

\section{Elimination of Pressure $p^-$}

As stated earlier, we aim to establish Hadamard well-posedness of the linearized coupled system given in (\ref{full_sys}) - (\ref{bcs}) for initial data $[\mathbf{u}^+_0, p^+(t = 0), \mathbf{u}^-_0]$ in the natural space of finite energy, $\mathcal{H} = \mathbf{L}^2(\Op) \times L^2(\Op) \times \{  \mathbf{f} \in \mathbf{L}^2(\Om): \div(\mathbf{f}) = 0 \text{ and } \mathbf{f} \cdot \vec{\nu} |_{\partial \Om \setminus \Gamma} = 0 \}$. Here, $\mathcal{H}$ is a Hilbert space with topology given by the following inner product: 
\beq
(\mathbf{y}_1, \mathbf{y}_2)_\mathcal{H} = (\mathbf{u}_1, \mathbf{u}_2)_{\Op} + (p_1, p_2)_{\Op} + (\mu_1, \mu_2)_{\Om}
\eeq
for any $\mathbf{y}_i = (\mathbf{u}_i, p_i, \mu_i) \in \mathcal{H}$, $i = 1,2$. Here, of course, we use the notation $(f,g)_{\Omega} = \int fg \, d\Omega$ (or vector-valued, as appropriate).

As a first step, the $p^-$ variable is eliminated by identifying it as the solution of a certain elliptic boundary value problem, which allows for an explicit semigroup generator formulation of the fluid-fluid model in (\ref{full_sys}) - (\ref{bcs}). In fact, the generator $\mathcal{A}: \mathcal{D}(\mathcal{A}) \subset \mathcal{H} \to \mathcal{H}$ is given explicitly in (\ref{genA}). 

Towards determining the appropriate boundary value problem to eliminate $p^-$, take the divergence of the first equation in (\ref{um_equns}), which gives
$$
(\div(\bum))_t - \Lap(\div(\bum)) + \Lap p^-  = 0\text{ in } \Omega^- \times (0,T).
$$

Combining this with $\div (\bum) = 0$ from (\ref{um_equns}), we have then \begin{equation}
\Lap p^- = 0  \text{ on } \Omega^- \times (0,T). 
\end{equation}

Now, the second equation in (\ref{bcs}) dotted with $\vec{\nu}$ on $\Gamma$ gives:
\beq 
p^- = \frac{\partial \bum}{\partial \vec{\nu}}  \cdot \vec{\nu} - [\sigma(\bup) \vec{\nu}] \cdot \vec{\nu} + p^+ \text{ on } \Gamma \times (0,T). 
\eeq

And finally, the first equation in (\ref{um_equns}) dotted with $\vec{\nu}$ and restricted to $\partial \Omega^- \setminus \Gamma$ gives
$$
\mathbf{u}_t^- \cdot \vec{\nu} - \Lap \bum \cdot \vec{\nu} + \grad p^- \cdot \vec{\nu} = 0 \text{ on }\partial \Omega^- \setminus \Gamma \times (0,T).
$$
Since $\bum = 0 \text{ on } \partial \Omega^- \setminus \Gamma$, it follows that $\mathbf{u}_t^- = 0$ on $\partial \Omega^- \setminus \Gamma \times (0,T)$ and so
\beq 
\frac{\partial p^-}{\partial \vec{\nu}} = \Lap \bum \cdot \vec{\nu} \text{ on } \partial \Omega^- \setminus \Gamma \times (0,T).
\eeq

So our boundary value problem to eliminate $p^-$ is
\begin{equation}\label{bvp_for_p}
\begin{cases}
\Lap p^- = 0 & \text{ on } \Omega^- \times (0,T),\\
p^- = \frac{\partial \bum }{\partial \vec{\nu}} \cdot \vec{\nu} - [\sigma(\bup) \vec{\nu}] \cdot \vec{\nu} + p^+  &\text{ on } \Gamma \times (0,T), \\
\frac{\partial p^-}{\partial \vec{\nu}} = \Lap \bum \cdot \vec{\nu} & \text{ on } \partial \Omega^- \setminus \Gamma \times (0,T).
\end{cases}
\end{equation}

Let the Dirichlet and Neumann maps, $D_s:{L}^2(\Gamma) \to {L}^2(\Omega^-)$ and \\$N_s: {L}^2\left(\partial \Omega^- \setminus \Gamma \right) \to {L}^2(\Omega^-)$, respectively, be given by

\beq \label{DirichletMap}
{h} = D_s({\varphi})  \iff \begin{cases}
\Lap {h} = {0} & \text{ on } \Omega^-,\\ 
{h} = {\varphi}  & \text{ on } \Gamma, \\
\frac{\partial {h} }{\partial \vec{\nu}} = {0} & \text{ on } \partial \Omega^- \setminus \Gamma,
\end{cases}
\eeq

and 
\beq 
{h} = N_s({\varphi})  \iff \begin{cases}
\Lap {h} = {0} & \text{ on } \Omega^-,\\ 
{h} = {0}  & \text{ on } \Gamma, \\
\frac{\partial {h} }{\partial \vec{\nu}} = {\varphi} & \text{ on } \partial \Omega^- \setminus \Gamma,
\end{cases}
\eeq
and so each of these maps gives a harmonic extension of boundary data.  
Note that both of these maps are well-defined by standard elliptic PDE theory (see \cite{Lasiecka&Triggiani}). Moreover, by elliptic regularity (see \cite{Lions&Magenes}), these maps also satisfy $D_s \in \mathcal{L}({H}^{-1/2}(\Gamma), {L}^2(\Om))$ and $N_s \in \mathcal{L}({H}^{-3/2}(\Gamma), {L}^2(\Om))$. 

With these maps in hand, we then have that the solution to (\ref{bvp_for_p}), for fixed $t \in (0,T)$ is given by
\beq \label{pm_cond}
p^-(t) = D_s \left(\frac{\partial \bum(t)}{\partial \vec{\nu}} \cdot \vec{\nu} - [\sigma(\bup(t)) \vec{\nu}] \cdot \vec{\nu} + p^+(t)  \right)  + N_s(\Lap \bum(t) \cdot \vec{\nu}) \text{ in } \Om.
\eeq

We claim that $p^-$, as defined in (\ref{pm_cond}), is in $L^2(\Om)$. To this end, we note the following fact. 

\begin{fact}

For ${f} \in {L}^2(\Om)$, if ${z}$ solves 
\beq
\begin{cases} \label{Grisvard_result}
\Lap {z} = {f} & \text{ in } \Om,\\
{z} |_\Gamma = {0} & \text{ on } \Gamma, \\
\frac{\partial {z}}{\partial \vec{\nu}}|_{\partial \Om \setminus \Gamma} = 0 & \text{ on } \partial \Om \setminus \Gamma,
\end{cases}
\eeq
then ${z} \in {H}^2(\Om)$  and $||{z}||_{{H}^2(\Om)} \leq c||{f}||_{{L}^2(\Om)}$ (see \cite{Grisvard2}, with rectangular domain).

\end{fact}

 Then, by the method of transposition (see \cite{Lions&Magenes}) and the availability of (\ref{Grisvard_result}), we have that, for $h \in H^{-1/2}(\Gamma)$ and $g \in H^{-3/2}(\partial \Om \setminus \Gamma)$, if $p^-$ satisfies 
\beq
\begin{cases}
\Lap p^- = 0 & \text{ in } \Om,\\
p^- = h & \text{ on } \Gamma, \\
\frac{\partial p^-}{\partial \vec{\nu}} = g & \text{ on } \partial \Om \setminus \Gamma, 
\end{cases}
\eeq
then $p^- \in L^2(\Om)$, giving the desired regularity of $p^-$. 

To show this regularity explicitly for the Dirichlet map, the lemma below parallels Lemma 6.1 in (\cite{AGW2}). Results for the Neumann map would follow similarly. 
\begin{lemma}
The Dirichlet map, $D_s$, as defined in (\ref{DirichletMap}), is an element of $\mathcal{L}(H^{-1/2}(\Gamma), L^2(\Om))$.

\end{lemma}

\begin{proof}
\textbf{Step 1.} We begin by defining the mixed Laplacian, $A_M: \mathcal{D}(A_M) \subset {L}^2(\Om) \to {L}^2(\Om)$ with homogeneous mixed boundary conditions by 
\beq 
A_M \, {v} = - \Lap {v}, \, \, {D}(A_M) = {H}^2(\Om) \cap {H}^1_{\Gamma}(\Om).
\eeq

As defined, $A_M$ is a positive definite, self-adjoint operator. Moreover, since the domain $\Om$ is rectangular, then by \cite{Grisvard2}, we have that 
\begin{center}
$A_M$ is an isomorphism from $\mathcal{D}(A_M)$ onto ${L}^2(\Om) $. 
\end{center}
In turn, by duality, we have that 
\begin{center} \label{isomorphism}
$A_M$ is an isomorphism from ${L}^2(\Om)$ onto $[\mathcal{D}(A_M)]'$. 
\end{center}

\textbf{Step 2.} Since $\Om$ has  a Lipschitz boundary, then for a given $v \in \mathcal{D}(A_M)$, its normal derivative $\left.\frac{\partial v}{\partial \vec{\nu}}\right|_{\Gamma}$ is only assured to be in ${L}^2(\partial \Om)$ (see \cite{necas2011direct}). However, from the recent result in \cite{BUFFA2001699}, we have that for any ${f} \in {H}^2(\Om)$, 
\beq
\grad_{\partial \Om} {f} + \frac{\partial {f}}{\partial \vec{\nu}} \vec{\nu} \in {H}^{1/2} (\partial \Om),
\eeq
where $[\grad_{\partial \Om} {f}]_{\partial \Om}$ denotes the tangential gradient of ${f}|_{\partial \Om}$ (see Theorem 5 of \cite{BUFFA2001699}). Since $\vec{\nu}|_\Gamma= [0,1]$, we then infer that, in particular, 
\beq  \label{inH12}
\left. \frac{\partial {v} }{\partial \vec{\nu}} \right|_{\Gamma} \in {H}^{1/2}(\Gamma), \, \text{ for every } {v} \in \mathcal{D}(A_M).
\eeq

\textbf{Step 3.} Given boundary function $\varphi \in {L}^2(\Gamma)$, we denote its extension by zero via 
$$
\varphi_{ext} = \begin{cases} 
0, &\text{ on } \partial \Om \setminus \Gamma\\
\varphi, &\text{ on } \Gamma.
\end{cases}
$$

Therewith, we define the linear functional $\ell_{\varphi}$, by having for any ${v} \in \mathcal{D}(A_M)$,
\begin{align}
\ell_\varphi({v}) & = \left( \varphi_{ext}, \frac{\partial {v}}{\partial \vec{\nu}} \right)_{\partial \Om} = \left( \varphi,  \frac{\partial {v}}{\partial \vec{\nu}} \right)_{\Gamma} \nonumber \\
& = \la \varphi,  \frac{\partial {v}}{\partial \vec{\nu}} \ra_{{H}^{-1/2}(\Gamma) \times {H}^{1/2}(\Gamma)},
\end{align}
where in the last equality, we are using (\ref{inH12}) and the fact that ${H}^{1/2} (\Gamma) = {H}^{1/2}_0(\Gamma)$ (see Theorem 3.40 in \cite{mclean2000strongly}). An estimation of this right hand side, via (\ref{inH12}) and the Closed Graph Theorem, then gives for a given $\varphi \in {L}^2(\Gamma)$, 
\beq 
|\ell_\varphi({v})| \leq C ||\varphi||_{{H}^{-1/2}(\Gamma)} ||{v}||_{\mathcal{D}(A_M)}.
\eeq
A subsequent extension by continuity gives that, for any $\varphi \in {H}^{-1/2}(\Gamma)$, $\ell_\varphi$ as given by 
\beq \label{inH12fullspace}
\ell_\varphi({v}) = \left( \varphi_{ext}, \frac{\partial {v}}{\partial \vec{\nu}} \right)_{\partial \Om} = \left< \varphi,  \frac{\partial {v}}{\partial \vec{\nu}} \right>_{{H}^{-1/2}(\Gamma) \times {H}^{1/2}(\Gamma)} \text{ for all } {v} \in \mathcal{D}(A_M),
\eeq
is an element of $[\mathcal{D}(A_M)]'$, with the estimate

\beq  \label{dualbound}
||\ell_\varphi||_{[\mathcal{D}(A_M)]'} \leq C ||\varphi||_{{H}^{-1/2}(\Gamma)}.
\eeq

\textbf{Step 4. } Via transposition, the problem of finding ${h} \in {L}^2(\Om)$ which solves the boundary value problem (\ref{DirichletMap}) with boundary data $\varphi \in {H}^{-1/2}(\Gamma)$ is the problem of finding ${h} \in {L}^2(\Om)$ which solves the relation 
\beq 
-({h}, A_M{v})_\Om = \ell_\varphi({v}), \, \, \text{ for every } {v} \in \mathcal{D}(A_M), 
\eeq
where $\ell_\varphi(\cdot)$ as given by (\ref{inH12fullspace}) is a well-defined element of $[\mathcal{D}(A_M)]'$, by Step 3. Accordingly, we can use (\ref{isomorphism}) and (\ref{dualbound}) to express the solution ${h} \in {L}^2(\Om)$ as 
\begin{align*}
&{h} = -A_M^{-1} \ell_\varphi \in {L}^2(\Om) \text{ with }\\
& ||{h}||_{{L}^2(\Om)} \leq ||A_M^{-1}||_{\mathcal{L}([\mathcal{D}(A_M)]', {L}^2(\Om))} ||\ell_\varphi||_{[\mathcal{D}(A_M)]'} \leq C_2 ||\varphi||_{{H}^{-1/2}(\Gamma)}.
\end{align*}

\end{proof}

With these Dirichlet and Neumann maps, $D_s$ and $N_s$, and their regularity established, we now define the linear maps $G_1$, $G_2$, and $G_3$ via
\beq
G_1\bum = - \grad \left( D_s\left( \frac{\partial \bum}{\partial \vec{\nu}} \cdot \vec{\nu} \right)  + N_s (\Lap \bum \cdot \vec{\nu}) \right), 
\eeq
\beq
G_2 \bup = - \grad \left( D_s([\sigma(\bup) \vec{\nu}] \cdot \vec{\nu} ) \right),
\eeq
\beq 
G_3 p^+ = - \grad(D_s(p^+)).
\eeq
Thus, from the first equation in  (\ref{um_equns}), we have that 
\begin{align}
\mathbf{u}_t^- & = \Lap \bum - \grad p^-  \nonumber \\
& = \Lap \bum + G_1 \bum + G_2 \bup + G_3 p^+ \text{ on } \Omega^+ \times (0,T).
\end{align}
 
\section{Main Result}

With the elimination of the $p^-$ variable, we now move towards establishing our main result. Bringing the first and second equations in (\ref{full_sys}) and (1.15) together yields the system
\beq 
\begin{cases}
\mathbf{u}_t^+ = - \bU \cdot \grad \bup + \div \,\sigma(\bup) - \grad p^+ & \text{ on } \Omega^+ \times (0,T),\\
p_t^+ = -\div(\bup) - \bU \cdot \grad p^+ & \text{ on } \Omega^+ \times (0,T),\\
\mathbf{u}_t^- = G_2 \bup + G_3 p^+ + \Lap \bum + G_1 \bum & \text{ on } \Omega^- \times (0,T).
\end{cases}
\eeq
This gives rise to the following system of equations
\begin{align}
\frac{d}{dt} \begin{bmatrix} \bup \\ p^+ \\ \bum \end{bmatrix} 
& = \bbm -\bU \cdot \grad \bup + \div \, \sigma(\bup) - \grad p^+ \\
- \div(\bup) - \bU \cdot \grad p^+ \\
G_2 \bup + G_3 p^+ + \Lap \bum + G_1\bum
\ebm \nonumber \\
& = \bbm -\bU \cdot \grad (\cdot) + \div \, \sigma(\cdot) &  - \grad (\cdot) & 0 \\
- \div(\cdot) & - \bU \cdot \grad(\cdot) & 0 \\
G_2 &  G_3 &  \Lap(\cdot) + G_1
\ebm 
\bbm \bup \\ p^+\\ \bum \ebm.
\end{align}

This is equivalent to the following system
\beq 
\frac{d}{dt}  \begin{bmatrix} \bup \\ p^+ \\ \bum \end{bmatrix} 
 = \mathcal{A}  \begin{bmatrix} \bup \\ p^+ \\ \bum \end{bmatrix}, 
\eeq

where \beq \label{genA} \mathcal{A} = \bbm -\bU \cdot \grad (\cdot) + \div \, \sigma(\cdot) &  - \grad (\cdot) & 0 \\
- \div(\cdot) & - \bU \cdot \grad(\cdot) & 0 \\
G_2 &  G_3 &  \Lap(\cdot) + G_1
\ebm. \eeq

Of course, if $\mathcal{A}: \mathcal{D}(\mathcal{A}) \subset  \mathcal{H} \to \mathcal{H}$ is to generate a $C_0$-semigroup of contractions - allowing for solutions $[\bup(t), p^+(t), \bum(t)]$ of (\ref{full_sys}) - (\ref{bcs}) to be obtained by applying $\{e^{\mathcal{A}t}\}$ to initial data $[\mathbf{u}_0^+, p^+(t=0), \mathbf{u}_0^-]  $ - then the domain of $\mathcal{A}$, $\mathcal{D}(\mathcal{A})$, should be constructed to allow for such a semigroup. In particular, we need $\mathcal{D}(\mathcal{A})$ so that $\mathcal{A}:  \mathcal{H} \to  \mathcal{H} $ is maximal dissipative with respect to $\mathcal{D}(\mathcal{A})$, allowing one to apply the Lumer-Phillips Theorem to obtain the $C_0$-semigroup of contractions. 

Prior to specifying the domain $\mathcal{D}(\A)$, we first need some preliminaries. The proofs of Lemma \ref{lemma:lemma1} and Lemma \ref{lemma:lemma2} can be found in \cite{Avalos&Dvorak}.  

\begin{lemma}\label{lemma:lemma1} \cite{Avalos&Dvorak} 
Suppose an $L^2(\Om)$-function $\rho$ satisfies $\Lap \rho \in L^2(\Om)$. Then one has the following boundary trace estimate:
\beq
\left| \left|\rho|_{\partial \Om} \right| \right|_{H^{-1/2}(\partial \Om)} +\left| \left|  \left. \frac{\partial \rho}{\partial \vec{\nu}} \right|_{\partial \Om}  \right| \right|_{H^{-3/2}(\partial \Om)}  \leq C\left\{ ||\rho||_{L^2(\Om)} + ||\Lap \rho||_{L^2(\Om)}\right\}.
\eeq
\end{lemma}

\begin{lemma}\cite{Avalos&Dvorak} \label{lemma:lemma2}
Suppose a pair $(\mu, \rho) \in \mathbf{L}^2(\Om) \times L^2(\Om)$ has the following properties:
\begin{enumerate}\item $\mu \in \mathbf{H}^1(\Om)$;
\item $\div\, (\mu) = 0$;
\item $-\Lap \mu + \grad \rho = \mathbf{h}$, where $\mathbf{h} \in \mathbf{L}^2(\Om)$ and $\div\, (\mathbf{h}) = 0$.
\end{enumerate}
Then one has the additional boundary regularity for the pair $(\mu, \rho)$:

\begin{minipage}{.45\textwidth}
\begin{align}
\rho|_{\partial \Om} &\in H^{-1/2}(\partial \Om), \nonumber\\
\left. \frac{\partial \mu}{\partial \vec{\nu}} \right|_{\partial \Om} & \in \mathbf{H}^{-1/2}(\partial \Om), \nonumber
 \end{align}
\end{minipage}
\begin{minipage}{.45\textwidth}
\begin{align}
 \frac{\partial \rho}{\partial \vec{\nu}} &\in H^{-3/2}(\partial \Om), \nonumber \\
 [(\Lap \mu)\cdot \vec{\nu}]_{\partial \Om}& \in \mathbf{H}^{-3/2}(\partial \Om).\nonumber
\end{align}

\end{minipage}
\end{lemma}

Now with the notation that $ \mathbf{H}^1_\Gamma(\Omega) = \{\mathbf{f} \in  \mathbf{H}^1(\Omega): \mathbf{f} = 0 \text{ on } \Gamma\}$ to specify where the value on the boundary is $0$, we require 
$\mathcal{D}(\A) \subset \mathcal{H}$ such that:

\begin{enumerate}
\item[(A.1)] $ [\bup, p^+, \bum]  \in \mathbf{H}^1_{\partial \Omega^+ \setminus \Gamma} (\Omega^+) \times L^2(\Omega^+) \times [\mathbf{H}^1_{\partial \Omega^- \setminus \Gamma}(\Omega^-)].$
\item[(A.2)] $\bU \cdot \grad p^+ \in {L}^2(\Omega^+)$.
\item[(A.3)] $- \div \, \sigma(\bup) + \grad p^+ \in \mathbf{L}^2(\Omega^+)$ (and so $\sigma(\bup) \cdot \vec{\nu} - p^- \vec{\nu} \in \mathbf{H}^{-1/2}(\partial \Omega^+)$).
\item[(A.4)] $\bup = \bum $ on $\Gamma$.
\item[(A.5)] For data $[\mathbf{u}_0^+, \mathbf{u}_0^-]$, there exists a corresponding ``pressure" $\pi_0^- \in L^2(\Omega^-)$ such that 

\begin{enumerate}
\item $[\mathbf{u}_0^-, \pi_0^-]$ satisfies $$-\Lap \mathbf{u}_0^- + \grad \pi_0^- \in  \{ \mathbf{f}  \in \mathbf{L}^2(\Om): \div(\mathbf{f}) = 0, \, \, [\mathbf{f} \cdot \vec{\nu}]_{\partial \Om \setminus \Gamma} = 0\}.$$

So by Lemma \ref{lemma:lemma2}, we have continuously\\
\begin{minipage}{.4\textwidth}
\begin{align}
\pi_0^-|_{\partial \Om} &\in H^{-1/2}(\partial \Om), \nonumber\\
\left. \frac{\partial \mathbf{u}_0^-}{\partial \vec{\nu}} \right|_{\partial \Om} & \in \mathbf{H}^{-1/2}(\partial \Om), \nonumber
 \end{align}
\end{minipage}
\begin{minipage}{.4\textwidth}
\begin{align}
 \frac{\partial \pi_0^-}{\partial \vec{\nu}} &\in H^{-3/2}(\partial \Om), \nonumber \\
 [(\Lap \mathbf{u}_0^-)\cdot \vec{\nu}]_{\partial \Om}& \in \mathbf{H}^{-3/2}(\partial \Om).\nonumber
\end{align}

\end{minipage}

\item $[\mathbf{u}_0^+, p_0^+, \mathbf{u}_0^-] $ and $\pi_0^-$ obey the following on $\Gamma$:
$$
\sigma(\mathbf{u}_0^+) \vec{\nu} - p_0^+ \vec{\nu} = \frac{\partial \mathbf{u}_0^-}{\partial \vec{\nu}} - \pi_0^- \vec{\nu} \, \, \text{ on } \Gamma.
$$
Thus, by the BVP (\ref{bvp_for_p}), we can identify this pressure $\pi_0^-$ associated with $[\mathbf{u}^+_0, p^+_0, \mathbf{u}^-_0]$ through the relation $$\pi_0^- = D_s \left(\frac{\partial \mathbf{u}^-_0}{\partial \vec{\nu}} \cdot \vec{\nu} - [\sigma(\mathbf{u}^+_0) \vec{\nu}] \cdot \vec{\nu} + p^+_0  \right)  + N_s(\Lap \mathbf{u}^-_0 \cdot \vec{\nu}),$$
from whence we obtain

$  \grad \pi_0^- = -G_1\mathbf{u}_0^- - G_2\mathbf{u}_0^+ - G_3p^+_0$.
\end{enumerate}

\end{enumerate}

With $\mathcal{D}(\mathcal{A})$ defined as such, we now present the main result. That is, we demonstrate semigroup wellposedness for\\ $\mathcal{A}: \mathcal{D}(\mathcal{A}) \subset \mathcal{H} \to \mathcal{H}$, the proof of which is based on the classic Lumer-Phillips Theorem. 

\begin{thm}
\begin{enumerate}
\item[(i)] The operator $\mathcal{A}: \mathcal{D}(\mathcal{A}) \subset  \mathcal{H} \to   \mathcal{H}$ is maximal dissipative. Therefore, by the Lumer-Phillips Theorem, it generates a $C_0$-semigroup of contractions $\{e^{\mathcal{A}t}\}_{t \geq 0}$ on $ \mathcal{H}$. 
\item[(ii)] Let $\l > 0$ and $[\mathbf{f}, g,  \mathbf{h}] \in  \mathcal{H}$ be given. (By part (i), there exists $[\bup, p^+, \bum] \in \mathcal{D}(\mathcal{A})$ which solves (\ref{resolventequn}).) Then $\bum$ and $p^-$   can be characterized as the solution to the variational system (\ref{compact_var_form}). The remaining unknown terms can then be characterized by 
\begin{align}
\bup &  = \mu_\lambda(\bum) + \tilde \mu([\mathbf{f}, g]^T) \nonumber \\
p^+ & = q_\lambda(\bum) + \tilde q([\mathbf{f}, g]^T) \nonumber,
\end{align}
where $[\mu_\l, q_\l]$ are defined in  (\ref{Dlambda}) and $[\tilde \mu, \tilde q]$ are defined in (\ref{psystem}).
\end{enumerate}
\end{thm}

\section{Proof of Theorem 1}

\subsection{Dissipativity}
In this section, we aim to show $\mathcal{R}e(\A \mathbf{v}, \mathbf{v})_H \leq 0$ for all $\mathbf{v} \in \mathcal{H}$. However, due to the non-zero ambient flow $\bU$, $\mathcal{A}$ as 
defined is not dissipative. Therefore, we introduce the bounded perturbation $\hat\A$ of the generator $\mathcal{A}$:
\beq
\hat\A = \A - \frac{\div(\bU)}{2} \bbm I & 0 & 0 \\
0 & I & 0 \\ 0 & 0 & 0 \ebm, \, \, \,  \mathcal{D}(\hat\A) = \mathcal{D}(\A).
\eeq
With this $\hat\A$, the standard perturbation result in \cite{Kato} can be applied, yielding semigroup generation for the original $\mathcal{A}$ of (\ref{genA}). So we aim to show that $\hat{\mathcal{A}}$ is maximal dissipative.

Considering the inner-product for the state-space  $\mathcal{H}$, for $\mathbf{v} = [\bup, p^+, \bum]^T \in \mathcal{D}(\hat\A)$, we have

\begin{align}
 \left( \hat \A \bbm \bup \\ p^+\\ \bum \ebm,  \bbm \bup \\ p^+\\ \bum \ebm \right)_\mathcal{H}  = & -(\bU \cdot \grad \bup, \bup)_{\Omega^+} + (\div \, \sigma(\bup), \bup)_{\Omega^+} - (\grad p^+, \bup)_{\Omega^+} \nonumber \\
 & - \frac12(\div(\bU) \bup,\bup)_{\Omega^+} -  (\div(\bup), p^+)_{\Omega^+} - (\bU\cdot \grad p^+, p^+)_{\Omega^+}\nonumber \\
 & - \frac12 (\div(\bU)p^+, p^+)_{\Omega^+}  + (\Lap \bum, \bum)_{\Omega^-} - (\grad p^-, \bum)_{\Omega^-}.
\end{align}

Applying Green's Theorem and the fact that $\bup\vert_{\partial \Omega^+ \setminus \Gamma} = 0$ to several terms, we have
\begin{itemize}
\item For the second term, $(\div \,  \sigma(\bup), \bup)_{\Omega^+} = - (\sigma(\bup), \e(\bup))_{\Omega^+} + \la \sigma(\bup) \vec{\nu}, \bup \ra_{\Gamma}.$
\item For the third term, we have $-(\grad p^+, \bup)_{\Omega^+} = (p^+, \div(\bup))_{\Omega^+} - \la p^+, \bup \cdot \vec{\nu} \ra_{\Gamma}$.
\item For the eighth term, $(\Lap \bum, \bum)_{\Omega^-} = - (\grad \bum, \grad \bum)_{\Omega^-} - \left\la \frac{\partial \bum}{\partial \vec{\nu}}, \bum \right\ra_{\Gamma}$, where we subtract, not add, the boundary term since $\vec{\nu}$ points \emph{into} $\Omega^-$.
\item For the ninth term, we have $-(\grad p^-, \bum)_{\Omega^-} = (p^-, \div(\bum))_{\Omega^-} + \la p^- \vec{\nu}, \bum \ra_\Gamma$, again adding the boundary term, not subtracting, because $\vec{\nu}$ points \emph{into} $\Om$. 
\end{itemize}

Substituting these in yields
\begin{align}
 \left( \hat \A \bbm \bup \\ p^+\\ \bum \ebm,  \bbm \bup \\ p^+\\ \bum \ebm \right)_\mathcal{H} & = -(\bU \cdot \grad \bup, \bup)_{\Omega^+} - (\sigma(\bup), \e(\bup))_{\Omega^+} + (p^+, \div(\bup))_{\Omega^+} \nonumber \\
& + \la \sigma(\bup)\vec{\nu} - p^+ \vec{\nu}, \bup \ra_{\Gamma} -  (\div(\bup), p^+)_{\Omega^+} - \frac12(\div(\bU) \bup,\bup)_{\Omega^+} \nonumber \\
 & - (\bU\cdot \grad p^+, p^+)_{\Omega^+}  - \frac12 (\div(\bU)p^+, p^+)_{\Omega^+} \nonumber \\
 &   - (\grad \bum, \grad \bum)_{\Omega^-} + (p^-, \div(\bum))_{\Omega^-} - \left\la\frac{\partial \bum}{\partial \vec{\nu}} -  p^- \vec{\nu}, \bum \right\ra_\Gamma.
\end{align}

Noting the boundary conditions in (\ref{bcs}), the boundary terms here cancel out. Furthermore, 
\begin{align*}
(p^+, \div(\bup))_{\Omega^+} -(\div(\bup), p^+)_{\Omega^+} & = (p^+, \div(\bup))_{\Omega^+} -\overline{(p^+, \div(\bup))_{\Omega^+}} \\
& = 2i\text{Im}(p^+, \div(\bup))_{\Omega^+}.
\end{align*}
Also, by virtue of the $\div(\bum) = 0$ condition, the $ (p^-, \div(\bum))_{\Omega^-}$ term drops out.

So we have 
\begin{align}  \left( \hat \A \bbm \bup \\ p^+\\ \bum \ebm,  \bbm \bup \\ p^+\\ \bum \ebm \right)_\mathcal{H} & = -(\bU \cdot \grad(\bup), \bup)_{\Omega^+} - (\sigma(\bup), \e(\bup))_{\Omega^+}  \nonumber \\
&+2i \text{Im} (p^+, \div(\bup))_{\Omega^+} - \frac12(\div(\bU) \bup,\bup)_{\Omega^+}  \nonumber \\
 &- (\bU\cdot \grad p^+, p^+)_{\Omega^+} - \frac12 (\div(\bU)p^+, p^+)_{\Omega^+}   - ||\grad \bum||^2_{\Omega^-}.\label{dissipativity}
\end{align}

For further simplification, now note that for functions $q, \tilde q \in L^2(\Omega^+)$ and with $\vec{\nu} = [\nu_1, \nu_2]^T$,
\begin{align*}
(\bU \cdot \grad q, \tilde q)_{\Op} & = \int_{\Omega^+} (\bU \cdot \grad q) \bar{\tilde q} \, d \Omega^+ \\
& = \int_{\Omega^+}  \left( \bbm U_1 \\ U_2 \ebm \cdot \bbm D_1q \\ D_2 q  \ebm  \right) \bar{\tilde q} \, d\Omega^+ \\
& =  \sum_{j=1}^2 \int_{\Omega^+} (U_j D_j q) \bar{\tilde q} \, d\Omega^+ \text{ (next applying Green's Theorem)} \\
& = \sum_{j=1}^2 \int_{\partial \Omega^+} U_j q \bnu_j \bar{\tilde q} \, d \partial \Omega^+  - \sum_{j=1}^2 \int_{\Omega^+} q D_j(U_j \bar{\tilde q} ) \, d \Omega^+\\
& = \int_{\partial \Omega^+} \cancelto{0 \text{ since } \bU \cdot \vec{\nu}|_{\partial \Omega^+} = 0}{(\bU \cdot \bnu)} q \bar{\tilde q} \, d \partial \Omega^+ \\
& - \left( \sum_{j=1}^2 \int_{\Omega^+} q \bar{\tilde q} Dj U_j \, d \Omega^+ + \sum_{j=1}^2 \int_{\Omega^+}q U_j D_j \bar{\tilde q} \, d \Omega^+ \right) \text{ (product rule)}\\
& = - \int_{\Omega^+} \div(\bU) (q \bar{\tilde q}) \, d\Omega^+ - \int_{\Omega^+} q \bU \cdot \grad \bar{\tilde q} \, d\Omega^+.
\end{align*}

Thus, if $q = \tilde q$, we have 
\begin{align}
(\bU \cdot \grad q, q)_{\Omega^+} & = - (\div(\bU), |q|^2)_{\Omega^+} - (q, \bU \cdot \grad q)_{\Omega^+}\nonumber \\
& = - (\div(\bU), |q|^2)_{\Omega^+} -\overline{(\bU \cdot \grad q, q)}_{\Omega^+}.
\end{align}

So adding $ \overline{(\bU \cdot \grad q, q)}_{\Omega^+}$ to both sides, we have 
\beq \label{2Re}
2 \mathcal{R}e(\bU \cdot \grad q, q)_{\Omega^+} = - (\div(\bU), |q|^2)_{\Omega^+}.
\eeq

Similarly, for $\bup \in \mathbf{H}^1(\Op)$, we have  
\beq\label{2Re_up}
2\mathcal{R}e(\bU \cdot \grad \bup, \bup)_{\Omega^+} = - (\div(\bU), |\bup|^2)_{\Omega^+}.
\eeq 

Furthermore, note that
\begin{align}\label{stress_strain_semi_pos_def}
(\sigma(\bup), \epsilon(\bup))_{\Op} & = (2 \nu \epsilon(\bup) + \tilde\lambda [I_2 \cdot \epsilon(\bup)]I_2, \epsilon(\bup))_{\Op}\nonumber\\
& = 2\nu ||\epsilon(\bup)||^2_{\Op} + \lambda( [I_2 \cdot \epsilon(\bup)]I_2, \epsilon(\bup))_{\Op}\nonumber\\
& = 2\nu ||\epsilon(\bup)||^2_{\Op} + \lambda(\div(\bup))^2 \nonumber\\
& \geq 0.
\end{align}

Thus, applying (\ref{2Re}), (\ref{2Re_up}), and (\ref{stress_strain_semi_pos_def}) to (\ref{dissipativity}), we have
\beq
\mathcal{R}e \left( \hat \A \bbm \bup \\ p^+\\ \bum \ebm,  \bbm \bup \\ p^+\\ \bum \ebm \right)_\mathcal{H}  =  - (\sigma(\bup),\e(\bup))_{\Omega^+} - ||\grad \bum||^2_{\Omega^-} \leq 0. 
\eeq

Thus, $\hat{\mathcal{A}}$ is dissipative.

\subsection{Maximality} \label{sec:Maximality}
In this section we show the maximality of the operator $\hat{\mathcal{A}}$ on the space $ \mathcal{H}$. To this end, we establish the \emph{range condition}, $Range(\lambda I - \hat\A) =  \mathcal{H}$ for parameter $\lambda$ sufficiently large. Our proof will make use of the classic Babuska-Brezzi Theorem, which we state here (see e.g., \cite{Kesavan}).

\begin{thm} (Babuska-Brezzi Theorem)  \label{thm:BBT}
Let $\Sigma$, $M$ be Hilbert spaces and $a: \Sigma \times \Sigma \to \R$ and $b: \Sigma \times M \to \R$ be bilinear forms which are continuous. Let
$$
Z = \{ \sigma \in \Sigma \, | \,  b(\sigma,q) = 0 \, \, \text{ for every } q \in M\}.
$$
Assume that $a(\cdot, \cdot)$ is $Z$-elliptic, i.e. there exists a constant $\alpha > 0$ such that 
$$
a(\sigma, \sigma) \geq \alpha ||\sigma||^2_\Sigma,
$$
for every $\sigma \in Z$. Assume further that there exists a constant $\beta > 0$ such that 
$$
\sup_{\tau \in \Sigma} \frac{b(\tau, q)}{||\tau||_\Sigma} \geq \beta ||q||_M \, \, \text{ for every } q \in M.
$$
Then if $\kappa \in \Sigma$ and $\ell \in M$, there exists a unique pair $(\sigma,p) \in \Sigma \times M$ such that 
\begin{align*}
a(\sigma, \tau) + b(\tau, p) & = (\kappa, \tau) \, \, \text{ for every } \tau \in \Sigma, \\
b(\sigma, q) & = (\ell, q) \, \, \text{ for every } q \in M.
\end{align*}

\end{thm}

Throughout, we will also make use of the Lax-Milgram Lemma, also stated here for convenience (see \cite{AB}). 

\begin{thm} (Lax-Milgram Lemma)
Let $V$ be a Hilbert space, $a: V \times V \to \R$ be a continuous bilinear form, and $b: V \to \R$ be a continuous linear form. Further assume that $a(\cdot, \cdot)$ is $V$-elliptic, i.e. there exists some constant $\alpha > 0$ such that  
$$
a(v,v) \geq \alpha ||v||_V^2 \, \, \text{ for all } v \in V. 
$$
Then there exists a unique $\hat u \in V$ such that $a(\hat u, v) = b(v) \text{ for all } v \in V$.
\end{thm}

Establishing the range condition is equivalent to finding some $[\bup, p^+, \bum] \in \mathcal{D}(\A)$ which satisfies, for arbitrary $[\mathbf{f},g,\mathbf{h}] \in \mathcal{H}$, the abstract equation 
\beq \label{resolventequn}
(\lambda I - \hat\A) \bbm \bup\\ p^+ \\ \bum \ebm = \bbm\mathbf{ f} \\ g\\ \mathbf{h} \ebm. 
\eeq

Component-wise, this gives the following system of equations:
\begin{align} \label{varequs}
 \l \bup + \bU \cdot \grad \bup - \div \,  \sigma(\bup) + \frac12 \div(\bU) \bup + \grad p^+ = \mathbf{f} & \text{ in } \Omega^+,\\
 \l p^+ + \div(\bup) + \bU \cdot \grad p^+ + \frac12 \div(\bU) p^+ = g & \text{ in } \Omega^+, \label{Ugradp+}\\
 \l \bum - \Lap \bum + \grad p^- = \mathbf{h} & \text{ in } \Omega^-.  \label{um_equn}
\end{align}

Further, since $[\bup, p^+, \bum] \in \mathcal{D}(\A)$, we have the additional relations
\begin{align}\label{more_resolvent}
\bup = 0 & \text{ on } \partial \Op \setminus \Gamma,\\
\bum = 0 & \text{ on } \partial \Om \setminus \Gamma,\\
\bup = \bum & \text{ on } \Gamma,\\
\sigma(\bup) \vec{\nu} - p^+ \vec{\nu} = \frac{\partial\bum}{\partial \vec{\nu}} - \pi^- \vec{\nu} & \text{ on }\Gamma,\\
\div(\bum) = 0 & \text{ in } \Om.\label{end_resolvent}
\end{align}

We aim to establish suitable bilinear forms in $\bum$ and $p^-$, with the goal of applying the Babuska-Brezzi Theorem. To this end, consider (\ref{um_equn}). Multiplying through by  $\varphi \in \mathbf{H}^1_{\partial \Omega^- \setminus \Gamma}(\Omega^-)$ and integrating over $\Om$, along with Green's Theorem and the boundary condition, gives

\beq \label{Lap_um_plus_gradpm}
(\l \bum , \varphi)_{\Omega^-} - (\Lap \bum, \varphi)_{\Om} + (\grad p^-, \varphi)_\Om  = (\mathbf{h}, \varphi)_\Om, 
\eeq
where $(\cdot, \cdot)_\Omega$ is the usual $\mathbf{L}^2$-inner product on $\Omega$ (or $L^2$-inner product, as appropriate), with $\la \cdot, \cdot \ra_{\Gamma}$ denoting the boundary integral on $\Gamma$.

A few application of Green's Theorem and then boundary conditions gives
\begin{align}
\l(\bum, \varphi)_\Om + (\grad \bum, \grad \varphi)_\Om - (p^-, \div( \varphi))_\Om + \left \la \frac{\partial \bum}{\partial \vec{\nu}} - p^- \vec{\nu}, \varphi \right\ra_\Gamma & = (\mathbf{h}, \varphi)_\Om,
\end{align}
or equivalently,
\begin{align}
 \l(\bum, \varphi)_\Om + (\grad \bum, \grad \varphi)_\Om - (p^-, \div( \varphi))_\Om + \left \la \sigma(\bup) - p^+ \vec{\nu}, \varphi \right\ra_\Gamma & = (\mathbf{h}, \varphi)_\Om. \label{seven}
\end{align}

We will consider the $\left \la \sigma(\bup) \vec{\nu} - p^+ \vec{\nu}, \varphi \right\ra_\Gamma$ term next, but first need to define some maps and results regarding them.

\begin{lemma} \label{Dlemma}
The solution map $D_\l: \mathbf{H}^{1/2}(\Gamma) \to \mathbf{H}^1_{\partial \Op \setminus \Gamma} (\Omega^+) \times L^2 (\Op)$ given by $$D_\l(\varphi) = \bbm\mu_\l(\varphi)\\ q_\l(\varphi)\ebm,$$ where
\begin{align} \label{Dlambda}
\begin{cases}
\l \mu_\l  + \bU \cdot \grad \mu_\l - \div\, \sigma(\mu_\l) + \frac12 \div(\bU) \mu_\l + \grad q_\l = 0 & \text{ in } \Op, \\
\l q_\l + \div(\mu_\l) + \bU \cdot \grad q_\l + \frac12 \div(\bU) q_\l = 0 & \text{ in } \Op, \\
\mu_\l\vert_{\Gamma} = \varphi|_\Gamma  & \text{ on } \Gamma, \\
\mu_\l \vert_{\partial \Op \setminus \Gamma} = 0 & \text{ on } \partial \Op \setminus \Gamma
\end{cases}
\end{align}
is wellposed, giving the existence and uniqueness of solutions $[\mu_\l(\varphi), \,  q_\l(\varphi)]$ in $ \mathbf{H}^1_{\partial \Op \setminus \Gamma} (\Omega^+) \times L^2 (\Op)$ for $\varphi$ in $\mathbf{H}^{1/2}(\Gamma)$. 
\end{lemma}

\begin{proof}

First, to resolve the ``pressure" term $q_\lambda$, we apply Theorem 9.1 in \cite{AGW}, which states, among other things, that for a given $\sigma \in L^2(\Op)$, there exists some $\gamma \in L^2(\Op)$ such that $\lambda \gamma + \bU \cdot \grad \gamma + \frac12\div(\bU)\gamma = \sigma $ is satisfied in the weak sense with $||\gamma||_\Op \leq \frac1\lambda ||\sigma||_\Op$ for a sufficiently large $\lambda$.

It is assumed $\mu_\lambda \in \mathbf{H}^1(\Op)$, so we have that $-\div(\mu_\lambda) \in \mathbf{L}^2(\Op)$. Thus, applying Theorem 9.1 to 
\beq  \label{qlambda}
\lambda q_\lambda + \bU \cdot \grad q_\lambda + \frac12\div(\bU) q_\lambda = -\div(\mu_\lambda) \, \, \text{ in } \Op
\eeq
 gives solution $q_\lambda[\mu_\lambda]$ with $$||q_\lambda||_\Op \leq \frac1\lambda ||\div(\mu_\lambda)||_\Op \leq \frac{C}\lambda ||\mu_\lambda||_{\mathbf{H}^1(\Op)}$$
for sufficiently large $\lambda = \lambda(\bU)$. 

Note that this only gives solution $q_\lambda[\mu_\lambda]$ in the weak sense; that is, for any $\varphi \in \mathbf{L}^2(\Op)$, we have 

\beq \label{q_lambda_var_form}
\lambda(q_\lambda, \varphi)_\Op +  (\bU \cdot \grad q_\lambda , \varphi)_\Op + \frac12(\div(\bU) q_\lambda, \varphi)_\Op = -(\div(\mu_\lambda), \varphi)_\Op.
\eeq 
But since this holds for all $\varphi \in \mathbf{L}^2(\Op)$, it holds in particular for all $\varphi \in \mathcal{D}(\Op)$. Thus, 
\beq\lambda q_\lambda + \bU \cdot \grad q_\lambda + \frac12\div(\bU) q_\lambda = -\div(\mu_\lambda)
\eeq
 a.e. in $\Op$. 

Next, $\varphi \in \mathbf{H}^1_{\partial \Om \setminus \Gamma}(\Om)$ and the Sobolev Trace map is surjective, so there exists some $\gamma_0^+ \in \mathcal{L}(\mathbf{H}^{1/2}(\Gamma), \mathbf{H}^1(\Om))$ such that $\gamma_0 \gamma_0^+(\mathbf{g}) = \mathbf{g}$ for all $\mathbf{g} \in \mathbf{H}^{1/2}(\Gamma)$. 
Furthermore, since $\Om$ and $\Op$ are symmetric across $\Gamma$, there exists some $\tilde \gamma_0^+ \in \mathcal{L}(\mathbf{H}^{1/2}(\Gamma), \mathbf{H}^1(\Op))$. (For example, if $\Gamma$ lies on the $x$-axis, let $\tilde \gamma_0^+ = - \gamma_0^+$.) 

Letting $\tilde \mu = \tilde \gamma_0^+(\varphi|_\Gamma) \in \mathbf{H}^1_{\partial \Op \setminus \Gamma} (\Op)$ with $\tilde \mu|_\Gamma = \varphi$, we now aim to find $\mu_0 \in \mathbf{H}^1_0(\Op)$ such that
\begin{align}
& \lambda(\mu_0, \psi)_\Op + (\bU \cdot \grad \mu_0, \psi)_\Op + (\sigma(\mu_0), \e(\psi))_\Op + \frac12(\div(\bU)\mu_0, \psi)_\Op \nonumber \\
& = (q_\lambda, \div(\psi))_\Op - (\lambda \tilde u + \bU \cdot \grad \tilde \mu+ \frac 12\div(\bU) \tilde \mu, \psi)_\Op - (\sigma (\tilde \mu), \e(\psi))_\Op 
\end{align}
for all $ \psi  \in \mathbf{H}^1_0(\Op)$.
With this $\mu_0$, we would then take $\mu_\lambda = \mu_0 + \tilde \mu$. 

To show the existence of such a $\mu_0$, we aim to apply the Lax-Milgram Lemma to $$a(\mathbf{u},\mathbf{v}) = F(\mathbf{v}),$$ where for all $\mathbf{u},\mathbf{v} \in \mathbf{H}^1_0(\Op)$, 
$a(\cdot, \cdot): \mathbf{H}^1_0(\Op) \times \mathbf{H}^1_0(\Op) \to \R$ is given by 
\beq
a(\mathbf{u},\mathbf{v}) = \lambda (\mathbf{u},\mathbf{v})_\Op + (\bU \cdot \grad \mathbf{u}, \mathbf{v})_\Op + (\sigma(\mathbf{u}), \e(\mathbf{v}))_\Op + \frac12(\div(\bU)\mathbf{u},\mathbf{v})_\Op
\eeq
and 
$F(\cdot): \mathbf{H}^1_0(\Op) \to \R$ is given by
\beq 
F(\mathbf{v}) = (q_\lambda, \div(\mathbf{v}))_\Op - (\lambda \tilde \mu + \bU \cdot \grad \tilde \mu + \frac 12 \div(\bU) \tilde \mu, \mathbf{v})_\Op  - (\sigma(\tilde\mu),\e(\mathbf{v}))_\Op.
\eeq

With a view of applying Lax-Milgram, note that by linearity of the integral and partial derivatives, $a(\cdot, \cdot)$ and $F(\cdot)$ are clearly linear.\\
Furthermore, $F(\cdot)$ is bounded. Indeed, for $\mathbf{v} \in \mathbf{H}^1_0(\Op)$, 
\begin{align}
|F(\mathbf{v})| & \leq|(q_\lambda, \div(\mathbf{v}))_\Op| + | (\lambda \tilde \mu + \bU \cdot \grad \tilde \mu + \frac12\div(\bU) \tilde \mu, \mathbf{v})_\Op| + |(\sigma(\tilde\mu), \e(v))_\Op| \nonumber \\
& \leq ||q_\lambda||_\Op ||\div(\mathbf{v})||_\Op + ||\lambda \tilde\mu + \bU \cdot \grad \tilde \mu + \frac12 \div(\bU)\tilde\mu||_\Op ||\mathbf{v}||_\Op \nonumber \\
& \hspace{.5in}+ ||\sigma(\tilde\mu)||_\Op|| \e(\mathbf{v})||_\Op \nonumber 
 \end{align}

Here, note that $||\div(\mathbf{v})||_\Op \leq ||\mathbf{v}||_{\mathbf{H}^1(\Op)}$ and $||\mathbf{v}||_\Op \leq ||\mathbf{v}||_{\mathbf{H}^1(\Op)}$. Additionally, $||\sigma(\tilde \mu)||_\Op \leq C||\tilde \mu||_{\mathbf{H}^1(\Op)}$ and $||\epsilon(\mathbf{v})||_\Op \leq \tilde C ||\mathbf{v}||_{\mathbf{H}^1(\Op)}$. 
Furthermore, $q_\l \in L^2(\Op)$ so $||q_\l||_\Op < \infty$. And finally, by Sobolev embedding, $\bU \in \mathbf{C}^1(\overline{\Op})$, so $||\bU||_\infty < \infty$.

Thus, we have 
\begin{align}
|F(\mathbf{v)}| & \leq C ( ||q_\lambda||_\Op + \lambda ||\tilde \mu||_\Op + ||\bU ||_\infty ||\tilde \mu||_{\mathbf{H}^1(\Op)} + ||\tilde \mu||_\Op + ||\tilde\mu||_{\mathbf{H}^1(\Op)} )||\mathbf{v}||_{\mathbf{H}^1(\Op)}  \nonumber  \\
& \leq C_{\tilde\mu, \bU, q_\lambda}||\mathbf{v}||_{\mathbf{H}^1(\Op)}.
\end{align}

 Also note that $a(\cdot, \cdot)$ is bounded. Indeed, for every $\mathbf{u},\mathbf{v} \in \mathbf{H}^1_0(\Op)$,
 \begin{align}
 a(\mathbf{u},\mathbf{v}) & = \lambda (\mathbf{u},\mathbf{v})_\Op + (\bU \cdot \grad \mathbf{u}, \mathbf{v})_\Op + (\sigma(\mathbf{u}), \e(\mathbf{v}))_\Op + \frac12(\div(\bU)\mathbf{u},\mathbf{v})_\Op \nonumber\\
 & \leq \lambda||\mathbf{u}||_\Op ||\mathbf{v}||_\Op + ||\bU||_\infty ||\grad \mathbf{u}||_\Op ||\mathbf{v}||_\Op + ||\sigma(\mathbf{u})||_\Op ||\e(\mathbf{v})||_\Op \nonumber\\
 & \hspace{.5in} + C ||\div(\bU) \mathbf{u}||_\Op||\mathbf{v}||_\Op \nonumber\\
 & \text{(by Cauchy-Schwarz  and using $\bU \in C^1(\overline{\Op})$ by Sobolev embedding)}\nonumber \\
 & \leq \tilde{C}( \lambda||\mathbf{u}||_{\mathbf{H}^1(\Op)} ||\mathbf{v}||_{\mathbf{H}^1(\Op)}  + ||\bU||_\infty || \mathbf{u}||_{\mathbf{H}^1(\Op)}  ||\mathbf{v}||_{\mathbf{H}^1(\Op)}    \nonumber \\
 & \hspace{.5in} + ||\mathbf{u}||_{\mathbf{H}^1(\Op)}  ||\mathbf{v}||_{\mathbf{H}^1(\Op)} + C_\bU ||\mathbf{u}||_{\mathbf{H}^1(\Op)} ||\mathbf{v}||_{\mathbf{H}^1(\Op)})  \\
 & \hspace{.5in} \text{(since $\mathbf{U} \in C^1(\overline{\Op})$, $\div(\mathbf{U})$ is bounded)}\nonumber\\ 
 & = M ||\mathbf{u}||_{\mathbf{H}^1(\Op)} ||\mathbf{v}||_{\mathbf{H}^1(\Op)} \text{ for some } M > 0.
 \end{align}

We also have that $a(\cdot, \cdot)$ is $\mathbf{H}^1_0(\Op)$-elliptic. Indeed, for any $\mathbf{u} \in \mathbf{H}_0^1(\Op)$, 
\begin{align}
a(\mathbf{u},\mathbf{u}) & = \lambda (\mathbf{u},\mathbf{u})_\Op + (\bU \cdot \grad \mathbf{u}, \mathbf{u})_\Op + (\sigma(\mathbf{u}), \e(\mathbf{u}))_\Op + \frac12(\div(\bU)\mathbf{u},\mathbf{u})_\Op \nonumber \\
& = \lambda ||\mathbf{u}||_{\Op}^2 + (\sigma(\mathbf{u}), \e(\mathbf{u}))_\Op + (\bU \cdot \grad \mathbf{u} + \frac12\div(\bU)\mathbf{u}, \mathbf{u})_\Op \nonumber \\
& =  \lambda ||\mathbf{u}||_{\Op}^2 + 2 \nu ||\epsilon(\mathbf{u})||_\Op^2 + \lambda (\div(\mathbf{u}))^2 + 0 \\
& \geq \tilde c ||\mathbf{u}||_{\mathbf{H}^1(\Op)}^2, \label{a_elliptic}
\end{align}
using Korn's Inequality and noting that by (\ref{2Re}),  $2(\bU \cdot \grad \mathbf{u},\mathbf{u})_{\Op}  = - (\div(\bU), |\mathbf{u}|^2)_{\Op}$, and so $(\bU \cdot \grad \mathbf{u} + \frac12 \div(\bU)\mathbf{u}, \mathbf{u})_\Op = 0$ for all $\mathbf{u} \in \mathbf{H}_0^1(\Op)$.

So by the Lax-Milgram Lemma, we have the existence and uniqueness of a solution $\mu_0 \in \mathbf{H}^1_0(\Op)$ to $a(\mu_0, \mathbf{v}) = F(\mathbf{v})$  for all $\mathbf{v} \in \mathbf{H}^1_0(\Op)$, with $||\mu_0||_{\mathbf{H}^1(\Op)} \leq C ||q_\lambda||_\Op$.

With this $\mu_0$ in hand, and with $\tilde \mu = \widetilde \gamma_0^+(\varphi|_\Gamma)$, we now show that $\mu_\lambda(\varphi|_\Gamma) = \mu_0 + \tilde \mu$ and $q_\lambda(\varphi|_\Gamma)$ satisfies (\ref{Dlambda}): 

From (\ref{qlambda}), the second equation in (\ref{Dlambda}) is clearly satisfied with $||q_\lambda||_\Op \leq \frac{c}{\lambda}||\mu_\lambda||_{\mathbf{H}^1(\Op)} $. 

As for satisfying the first equation in (\ref{Dlambda}), note that for any $\mathbf{v}\in {\mathbf{H}^1_0(\Op)} $, we have, via Green's Identities and boundary terms dropping out since $\mathbf{v} \in \mathbf{H}_0^1(\Op)$, 
\begin{align}
& a(\mu_\lambda,\mathbf{v}) = F(\mathbf{v}) \nonumber\\
\Rightarrow & (\lambda \mu_0 + \bU \cdot \grad \mu_0 + \frac12\div(\bU)\mu_0,\mathbf{v})_\Op +  (\sigma(\mu_0), \e(\mathbf{v}))_\Op = \nonumber\\
&  (q_\lambda, \div(\mathbf{v}))_\Op - (\lambda \tilde\mu + \bU \cdot \grad \tilde\mu + \frac12\div(\bU) \tilde\mu, \mathbf{v})_\Op - (\sigma(\tilde\mu), \e(\mathbf{v}))_\Op \nonumber \\
\Rightarrow & (\lambda (\mu_0 + \tilde\mu)+ U \cdot \grad( \mu_0 + \tilde \mu)+ \frac12\div(\bU) (\mu_0 + \tilde \mu), \mathbf{v})_\Op +  (\sigma(\mu_0 + \tilde\mu), \e(\mathbf{v}))_\Op = \nonumber\\
&  (q_\lambda, \div(\mathbf{v}))_\Op \nonumber.
\end{align}
Setting $\mu_0 + \tilde\mu = \mu_\lambda$, this gives 
\beq \label{mu_lambda_var_form}
(\lambda \mu_\lambda + \bU\cdot \grad \mu_\lambda  + \frac12 \div(\bU) \mu_\lambda, \mathbf{v})_\Op - (\div \, \sigma(\mu_\lambda), \mathbf{v})_\Op + (\grad q_\lambda, \mathbf{v}) = 0.
\eeq

Since this holds for all $\mathbf{v} \in \mathbf{H}^1_0(\Op)$, it holds in particular for $\mathbf{v} \in \mathcal{D}(\Op)$. So we infer that the corresponding solution $\mu_\lambda$ satisfies (\ref{Dlambda}) pointwise a.e. in $\Op$. 

Finally, note that $\mu_\lambda|_\Gamma = 0 + \varphi$ and $\mu_{\partial \Op \setminus \Gamma} = 0 + 0$, so the boundary conditions of (\ref{Dlambda}) are satisfied. 

Thus, the map $D_\lambda(\cdot)$, as defined in (\ref{Dlambda}) is wellposed, haivng unique solutions $[\mu_\lambda(\varphi), q_\lambda(\varphi)]^T$, with continuous dependence on data $\varphi\in \mathbf{H}^{1/2}(\Gamma)$, i.e. 
\beq 
||\mu_\l(\varphi)||_{\mathbf{H}^1(\Op)} + ||q_\l(\varphi)||_{\Op} \leq C ||\varphi||_{\mathbf{H}^{1/2}(\Gamma)}.
\eeq

\end{proof}

We provide a similar result for the map $\mathbb{A}_\l$. 

\begin{lemma} \label{A_lemma}
For a given $\l > 0$, the map $\mathbb{A}_\lambda: \mathbf{H}_0^1(\Op) \times L^2(\Op) \to \mathbf{L}^2(\Op) \times L^2(\Op)$, given by 
\beq \label{AGWStep1}
\mathbb{A}_\l(\tilde \mu, \tilde q)  = \bbm
\l \tilde\mu + \bU \cdot \grad \tilde\mu - \div \, \sigma(\tilde\mu) + \frac12\div(\bU) \tilde\mu + \grad \tilde q \\
\l \tilde q + \div(\tilde\mu) + \bU \cdot \grad \tilde q + \frac12 \div(\bU) \tilde q
\ebm
\eeq 
 admits of a bounded inverse on $\mathbf{L}^2(\Op) \times L^2(\Op)$.
\end{lemma}

\begin{proof}
We aim to show that, for $[\mathbf{f}, g] \in \mathbf{L}^2(\Op) \times L^2(\Op)$, $\mathbb{A}_\l(\tilde \mu, \tilde q)   = [\mathbf{f}, g]^T$ has a unique, bounded solution.
Consider the system
\beq \label{psystem}\begin{cases}
\l \tilde\mu + \bU \cdot \grad \tilde\mu - \div \, \sigma(\tilde\mu) + \frac12\div(\bU) \tilde\mu + \grad \tilde q = \mathbf{f} & \text{ in } \Op, \\
\l \tilde q + \div(\tilde\mu) + \bU \cdot \grad \tilde q + \frac12 \div(\bU) \tilde q = g & \text{ in }\Op,\\
\tilde \mu \vert_{\partial \Op} = \mathbf{0}& \text{ on } \partial \Op.
\end{cases}
\eeq

To resolve the ``pressure" term, $\tilde q$, we again apply Theorem 9.1 in \cite{AGW}, which states that there exists some $\gamma \in L^2(\Op)$ such that \beq \label{Thm9.1result} \l \gamma + \bU \cdot \grad \gamma + \frac12 \div(\bU) \gamma = \sigma \in L^2(\Op) \eeq 
with $||\gamma||_{\Op} \leq \frac1\lambda || \sigma||_\Op$ for a sufficiently large $\lambda$. 

With $\tilde \mu \in \mathbf{H}^1_0(\Op)$, note that we can apply this twice, once with $\sigma = g$ and once with  $\sigma = \div(\tilde \mu)$ (both in $L^2(\Op)$) to obtain solutions  $\tilde q[g]$ and $\tilde q[\tilde\mu]$ to (\ref{Thm9.1result}) with 
\beq \label{qtildebound}
||\tilde q [g]||_{\Op} \leq \frac1\l ||g||_\Op
\eeq
and 
\beq \label{qtilde_mu_bound}
||\tilde q[\tilde\mu]||_\Op \leq \frac 1\lambda ||\div (\tilde \mu)||_\Op \leq \frac{c}\l ||\tilde\mu||_{\mathbf{H}^1(\Op)}.
\eeq
Thus, let $\tilde q = \tilde q [g] + \tilde q[\tilde \mu]$. 

Again note that this gives solution $\tilde q$ in the weak sense; that is, for any $\varphi \in \mathbf{L}^2(\Op)$, we have 
\beq \label{resolve_qg}
\lambda (\tilde q[g], \varphi)_\Op + (\bU \cdot \grad \tilde q[g], \varphi)_\Op + \frac12 (\div(\bU) \tilde q[g], \varphi)_\Op = (g, \varphi)_\Op.
\eeq
and \beq \label{resolve_qdivmu}
\lambda (\tilde q[\tilde \mu], \varphi)_\Op + (\bU \cdot \grad \tilde q[\tilde \mu], \varphi)_\Op + \frac12 (\div(\bU) \tilde q[\tilde \mu], \varphi)_\Op = (\div(\tilde \mu), \varphi)_\Op.
\eeq
But since this holds for all $\varphi \in \mathbf{L}^2(\Op)$, it holds in particular for all $\varphi \in \mathcal{D}(\Op)$. Thus,
\beq
\lambda \tilde q[g] + \bU \cdot \grad \tilde q[g] + \frac 12 \div (\bU)\tilde q[g] = g
\eeq
 and 
 \beq \lambda \tilde q[\tilde \mu] + \bU \cdot \grad \tilde q[\tilde \mu] + \frac 12 \div (\bU)\tilde q[\tilde \mu] = \div(\tilde \mu)
 \eeq
  a.e. in $\Op$.

We then define the bilinear form $\tilde a(\cdot, \cdot): \mathbf{H}^1_0(\Op) \times \mathbf{H}^1_0(\Op) \to \R$ as 
\begin{align}
 \tilde a(\mathbf{v}, \psi) & = (\l \mathbf{v} + \bU \cdot \grad \mathbf{v} - \div \, \sigma(\mathbf{v}) + \frac12 \div(\bU) \mathbf{v} + \grad \tilde q[\mathbf{v}], \psi)_\Op \nonumber \\
 & = (\lambda \mathbf{v} + \bU \cdot \grad \mathbf{v} + \frac12 \div(\bU)\mathbf{v}, \psi)_\Op + (\sigma(\mathbf{v}), \e(\psi))_\Op  \nonumber \\
 & \hspace{.5in} - \la \sigma(\mathbf{v}) \cdot \vec{\nu}, \psi \ra_{\partial \Op} - (\tilde q [\mathbf{v}], \div(\psi))_\Op + \la \tilde q[\mathbf{v}], \psi\cdot \bnu \ra_{\partial \Op}\nonumber \\
 & = (\lambda \mathbf{v} + \bU \cdot \grad \mathbf{v} + \frac12 \div(\bU)\mathbf{v}, \psi)_\Op + (\sigma(\mathbf{v}), \e(\psi))_\Op   - (\tilde q [\mathbf{v}], \div(\psi))_\Op  
 \end{align}
for all $ \mathbf{v}, \psi \in \mathbf{H}_0^1(\Op).$

Also let linear form $\tilde f: \mathbf{H}_0^1(\Op) \to \R$ be given by 
\begin{align}
\tilde f(\psi) & = (\mathbf{f},\psi)_\Op + ( \tilde q[g], \div(\psi))_\Op - \la \tilde q[g], \psi\cdot \bnu\ra_{\partial \Op} \\
& =  (\mathbf{f},\psi)_\Op + ( \tilde q[g], \div(\psi))_\Op  \text{ for all } \psi \in \mathbf{H}_0^1(\Op).
\end{align}
So the variational formulation is: 

Find $\tilde\mu \in \mathbf{H}_0^1(\Op)$ such that, for all $\psi \in \mathbf{H}_0^1(\Op)$, 
\beq \label{VarFormLaxMil}
\tilde a(\tilde \mu, \psi) = \tilde f( \psi).
\eeq

With the aim of applying Lax-Milgram to (\ref{VarFormLaxMil}), note that:
\begin{itemize}
\item $\tilde a( \cdot, \cdot)$ is clearly bilinear and $\tilde f(\cdot)$ is linear, due to the linearity of the integral. 
\item $\tilde a(\cdot, \cdot)$ is bounded: For $\mathbf{v}, \psi \in \mathbf{H}_0^1 (\Op)$ and then applying the Triangle Inequality and Cauchy-Schwarz Inequality,
\begin{align}
\tilde a(\mathbf{v}, \psi) & = (\lambda \mathbf{v} + \bU \cdot \grad \mathbf{v} + \frac12 \div(\bU)\mathbf{v}, \psi)_\Op + (\sigma(\mathbf{v}), \e(\psi))_\Op   - (\tilde q [\mathbf{v}], \div( \psi))_\Op \\
 & \leq \left( \l ||\mathbf{v}||_\Op + ||\bU||_{L^\infty(\Op)}||\grad \mathbf{v}||_\Op + \frac12\vert\vert\div(U)\mathbf{v} \vert\vert_{\Op} \right) ||\psi||_\Op \nonumber \\
& \hspace{.5in} + ||\sigma(\mathbf{v})||_\Op ||\e(\psi)||_\Op + ||\tilde q[\mathbf{v}] ||_\Op ||\div(\psi)||_\Op \\
& \leq \l ||\mathbf{v}||_{\mathbf{H}^1_0(\Op)}||\psi||_{\mathbf{H}^1_0(\Op)}  + C_U||\mathbf{v}||_{\mathbf{H}_0^1(\Op)}||\psi||_{\mathbf{H}^1_0(\Op)} \nonumber \\
&  \hspace{.5in}+ \tilde C ||\div(\bU)||_{L^\infty(\Op)} ||\mathbf{v}||_{\mathbf{H}^1_0(\Op)} ||\psi||_{\mathbf{H}^1_0(\Op)} \nonumber \\
& \hspace{.5in}+ ||\mathbf{v}||_{\mathbf{H}^1_0(\Op)} ||\psi||_{\mathbf{H}^1_0(\Op)} + ||\tilde q[\mathbf{v}] ||_{\mathbf{H}^1_0(\Op)} ||\psi||_{\mathbf{H}^1_0(\Op)}\\
& \leq C_1 ||\mathbf{v}||_{\mathbf{H}^1_0(\Op)}||\psi||_{\mathbf{H}^1_0(\Op)},
\end{align}

where $\vert\vert\div(\bU)\mathbf{v} \vert\vert_{\Op} \leq ||\div(\bU)||_{L^\infty(\Op)} ||\mathbf{v}||_\Op \leq C_\bU ||\mathbf{v}||_\Op$ as Sobolev embedding gives $ \bU \in \mathbf{H}^2(\Op) \hookrightarrow \mathbf{C}^1(\overline{\Op})$. 

\item $\tilde a( \cdot, \cdot)$ is $\mathbf{H}_0^1(\Op)$-elliptic:
First note that, as in (\ref{2Re}) and considering only $\R$-valued functions, $2(\bU \cdot \grad \psi, \psi)_{\Op} = - (\div(\bU), |\psi|^2)_{\Op}$, and so $(\bU \cdot \grad \psi + \frac12 \div(\bU)\psi, \psi)_\Op = 0$ for all $\psi \in \mathbf{H}_0^1(\Op)$.  

Also, 
\begin{align}
\left|(\tilde q[\psi], \div(\psi))_\Op \right| & \leq || \tilde q[\psi]||_\Op ||\div(\psi)||_\Op \nonumber\\
& \leq \left( \frac1 \l ||\psi||_{\mathbf{H}_0^1(\Op)} \right) ||\psi||_{\mathbf{H}^1(\Op)}  \text{ (by }(\ref{qtildebound}))\nonumber\\
& \leq \frac{c}\l ||\psi||_{\mathbf{H}^1(\Op)}^2.
\end{align}  

Also note that 
$$ (\sigma(\psi), \e(\psi))_{\Op} = 2 \bnu ||\e(\psi)||^2_{\Op} + \l(\div(\psi))^2.$$

And finally, $$\lambda ||\psi||^2_{\Op} + 2 \bnu||\e(\psi)||^2_\Op \leq \tilde c ||\psi||^2_{\mathbf{H}^1(\Op)}.$$

Putting all of this together gives
\begin{align}
\tilde a( \psi, \psi ) & = \lambda(\psi, \psi)_{\Op} + (\sigma(\psi), \e(\psi))_\Op - (\tilde q[\psi], \div(\psi))_\Op \nonumber\\
& =  \l ||\psi||^2_{\Op} + 2 \bnu ||\e(\psi)||^2_{\Op} + \l(\div(\psi))^2- (\tilde q[\psi], \div(\psi))_\Op \nonumber \\
& \geq \lambda ||\psi||_\Op^2 + 2 \bnu||\e(\psi)||_\Op^2 + 0 - \frac{c}\l ||\psi||_{\mathbf{H}^1(\Op)} \nonumber\\
& \geq \tilde c ||\psi||_{\mathbf{H}^1(\Op)}^2  - \frac{c}\l ||\psi||_{\mathbf{H}^1(\Op)} \text{ (by Korn's Inequality)} \nonumber \\
& \geq \frac{\tilde{c}}{2} ||\psi||_{\mathbf{H}^1(\Op)}^2 \text{ for sufficiently large $\l$},
\end{align}
which establishes the ellipticity of $\tilde a(\cdot, \cdot)$.
\end{itemize}
Thus, by Lax-Milgram, there exists a unique solution $\tilde\mu \in \mathbf{H}_0^1(\Op)$ to (\ref{VarFormLaxMil}). 

So with this $\tilde \mu$ in hand and setting $\tilde q = \tilde q[g] + \tilde q[ \tilde\mu] \in L^2(\Op)$, we now show that $\tilde \mu$ and $\tilde q$ satisfy (\ref{psystem}):

From (\ref{resolve_qg}) and (\ref{resolve_qdivmu}), the second equation of (\ref{psystem}) is clearly satisfied with 
$$
||\tilde q||_\Op \leq \frac{c}{\l} (||g||_\Op + ||\tilde \mu||_{\mathbf{H}^1(\Op)}).
$$
As for satisfying the first equation in (\ref{psystem}), note that for any $\psi \in \mathbf{H}^1_0(\Op)$, we have, via Green's Identities and boundary terms dropping out,
\begin{align}
&\tilde a(\tilde \mu, \psi) = \tilde f(\psi) \nonumber \\
& \Rightarrow  (\lambda \tilde \mu + \bU \cdot \grad \tilde \mu + \frac12 \div(\bU)\tilde \mu, \psi)_\Op + (\sigma(\tilde \mu), \e(\psi))_\Op   - (\tilde q [\tilde \mu], \div(\psi))_\Op \\
& =  (\mathbf{f},\psi)_\Op + ( \tilde q[g], \div(\psi))_\Op .
\end{align}

Setting $\tilde q = \tilde q[\mathbf{v}] + \tilde q[g]$, this gives
\beq \label{tilde_mu_var_form}
(\lambda \tilde \mu+ \bU \cdot \grad \tilde \mu + \frac12 \div(\bU)\tilde \mu, \psi)_\Op + (\sigma(\tilde \mu), \e(\psi))_\Op   - (\tilde q, \div(\psi))_\Op \\
 =  (\mathbf{f},\psi)_\Op.
\eeq
Since this holds for all $\psi \in \mathbf{H}_0^1(\Op)$, it holds in particular for $\psi \in \mathbf{C}_C^\infty(\Op)$. So we infer that the corresponding solution $\tilde \mu$ satisfies (\ref{psystem}) a.e. in $\Op$.

To show the boundedness, note that Lax-Milgram gives 
\begin{align} \label{mu_tilde_bound}
||\tilde \mu||_{\mathbf{H^1(\Op)}} & \leq C (||\mathbf{f}||_{\mathbf{L}^2(\Op)}|| + ||\tilde q[g]||_\Op) \nonumber \\
& \leq C (||\mathbf{f}||_{\mathbf{L}^2(\Op)}|| + \frac1\l ||g||_\Op),
\end{align}
after an application of (\ref{qtildebound}). Additionally, using (\ref{qtildebound}) and (\ref{qtilde_mu_bound}), 
\begin{align} \label{q_tilde_bound}
||\tilde q||_{\Op} & \leq ||\tilde q[g]||_\Op + ||\tilde q(\tilde \mu)||_\Op \nonumber \\
& \leq C||g||_\Op + \frac{c}{\l} ||\tilde \mu||_{\mathbf{H}^1(\Op)} \nonumber\\
& \leq C||g||_\Op + C_\l (||\mathbf{f}||_{\mathbf{L}^2(\Op)}|| +  ||g||_\Op)\nonumber \\
& \leq \tilde C (||\mathbf{f}||_{\mathbf{L}^2(\Op)}|| +  ||g||_\Op),
\end{align}
for some $\tilde C = \tilde C(\l) > 0$.

Thus, (\ref{mu_tilde_bound}) - (\ref{q_tilde_bound}) gives the boundedness of inverse map, $\mathbb{A}_\l^{-1}$.

\end{proof}

And now returning to our maximality argument. 

With this $D_\lambda$ map of (\ref{Dlambda}) in hand and applying Green's Theorem to each term, we have for $\varphi \in \mathbf{H}^1_{\partial \Om \setminus \Gamma}(\Om)$, 
\begin{align} \label{sigma_on_bound}
&\left \la \sigma(\bup) \vec{\nu} - p^+ \vec{\nu}, \varphi \right\ra_\Gamma \nonumber\\
&= (\div \, \sigma(\bup), \mu_\l(\varphi|_{\Gamma}))_{\Op}  + (\sigma(\bup), \e(\mu_\l(\varphi|_{\Gamma})))_{\Op} - (p^+, \div(\mu_\l(\varphi|_\Gamma)))_{\Op}  \nonumber \\
& \hspace{.5in} - (\grad p^+, \mu_\l(\varphi|_\Gamma))_{\Op} .
\end{align}

Using (\ref{varequs}) to rewrite the first term, this becomes
\begin{align}\label{ten}
 & (\lambda \bup + \bU \cdot \grad \bup + \frac 12 \div(\bU)\bup + \grad p^+, \mu_\l(\varphi|_\Gamma))_{\Op} - (\mathbf{f},\mu_\l(\varphi|_\Gamma))_{\Op} \nonumber\\
&  \hspace{.5in}+ (\sigma(\bup), \e(\mu_\l(\varphi|_\Gamma)))_{\Op} - (p^+, \div(\mu_\l(\varphi|_\Gamma)))_{\Op} - (\grad p^+, \mu_\l(\varphi|_\Gamma))_{\Op} \nonumber\\
& = (\lambda \bup + \bU \cdot \grad \bup + \frac 12 \div(\bU) \bup, \mu_\l(\varphi|_\Gamma))_{\Op} - (\mathbf{f},\mu_\l(\varphi|_\Gamma))_{\Op} \nonumber\\
&  \hspace{.5in}+ (\sigma(\bup), \e(\mu_\l(\varphi|_\Gamma)))_{\Op} - (p^+, \div(\mu_\l(\varphi|_\Gamma)))_{\Op}.
 \end{align}

Now to eliminate $\bup$ and $p^+$, consider the following BVP:
\beq \label{Op_problem}
\begin{cases}
\l \bup + \bU\cdot \grad \bup - \div \, \sigma(\bup) + \frac 12 \div(\bU) \bup + \grad p^+ = \mathbf{f} & \text{ in } \Op, \\
\l p^+ + \div(\bup) + \bU \cdot \grad p^+ + \frac12 \div(\bU) p^+ = g & \text{ in } \Op,\\
\bup = 0 & \text{ on } \partial \Op \setminus \Gamma, \\
\bup = \bum &\text{ on } \Gamma, \\
\sigma(\bup) \vec{\nu} - p^+ \vec{\nu} = \frac{\partial \bum}{\partial \vec{\nu}} - p^- \vec{\nu} & \text{ on } \Gamma.
\end{cases}
\eeq

Furthermore, let $\bbm \tilde \mu \\ \tilde q \ebm = \bA^{-1}_\l\bbm \mathbf{f} \\ g \ebm \in \mathbf{H}^1_0(\Op) \times L^2(\Op)$ for $\mathbf{f} \in \mathbf{L}^2(\Op)$ and \\
 $g \in L^2(\Op)$. The existence of such $\tilde \mu$ and $\tilde q$ was shown in Lemma \ref{A_lemma}. 

Then we see combining the maps $D_\lambda$ and $\mathbb{A}_\lambda$,  we have 
\begin{align}\label{up,pp} \bbm \bup \\ p^+ \ebm & = D_\l(\bum|_{\Gamma}) + \bA^{-1}_\l\left([\mathbf{f},  g ]^T \right)\nonumber,\\
& = \bbm \mu_\l(\bum|_{\Gamma}) \\ q_\l(\bum|_{\Gamma}) \ebm + \bbm \tilde u(f,g) \\ \tilde q(f,g) \ebm.
\end{align}
which establishes a portion of  Theorem 1(ii). 

Combining  (\ref{seven}) and (\ref{ten}) gives
\begin{align}
& \l(\bum, \varphi)_\Om + (\grad \bum, \grad \varphi)_{\Om} - (p^-, \div(\varphi))_\Om + [ \l(\bup, \mu_\l(\varphi|_\Gamma))_\Op\nonumber  \\
&  \hspace{.5in}+ (\bU \cdot \grad \bup, \mu_\l(\varphi|_\Gamma))_\Op + \frac12(\div(\bU) \bup, \mu_\l(\varphi|_\Gamma))_\Op \nonumber\\
& \hspace{.5in}+ (\sigma(\bup), \e(\mu_\l(\varphi|_\Gamma)))_\Op - (p^+, \div(\mu_\l(\varphi|_\Gamma)))_\Op - (\mathbf{f}, \mu_\l(\varphi|_\Gamma))_{\Op} ]  = (\mathbf{h},\varphi)_\Om
\end{align}
for $\varphi \in\mathbf{H}^1_{\partial\Om \setminus \Gamma} (\Om)$. 

Now using $\bup = \mu_\l(\bum|_{\Gamma}) + \tilde \mu(\mathbf{f},g)$ and $p^+ = q_\l(\bum|_{\Gamma}) + \tilde q(\mathbf{f},g)$ and subtracting the $\tilde \mu(\mathbf{f},g)$, $\tilde q(\mathbf{f},g)$ terms to the right hand side, we have for all $\varphi \in \mathbf{H}^1_{\partial\Om \setminus \Gamma} (\Om)$,

\begin{align}
&\l(\bum, \varphi)_\Om + (\grad \bum, \grad \varphi)_\Om - (p^-, \div(\varphi))_\Om + \l(\mu_\l (\bum|_\Gamma), \mu_\l(\varphi|_\Gamma))_\Om  \nonumber \\
& \hspace{-.01in}+ (\bU \cdot \grad \mu_\l(\bum|_\Gamma))_\Op + \frac12 (\div(\bU)\mu_\l(\bum|_\Gamma), \mu_\l(\varphi|_\Gamma))_\Op + (\sigma(\mu_\l(\bum|_\Gamma)), \e(\mu_\l(\varphi|_\Gamma)))_\Op \nonumber \\
& - (q_\l (\bum|_\Gamma), \div(\mu_\l(\varphi|_\Gamma)))_\Op \nonumber\\
& = (\mathbf{h}, \varphi)_\Om + (\mathbf{f}, \mu_\l(\varphi|_\Gamma))_\Op - \big[ \l(\tilde\mu(\mathbf{f},g), \mu_\l(\varphi|_\Gamma)))_\Op + (\bU \cdot \grad \tilde\mu(\mathbf{f},g), \mu_\l(\varphi|_\Gamma))_\Op \nonumber \\
& + \frac12(\div(\bU) \tilde\mu(\mathbf{f},g) , \mu_\l(\varphi|_\Gamma))_\Op + (\sigma(\tilde\mu(\mathbf{f},g)), \e(\mu_\l(\varphi|_\Gamma)))_\Op \nonumber \\
& - (\tilde q(\mathbf{f},g), \div(\mu_\l(\varphi|_\Gamma)))_\Op \big].
\end{align}

This becomes 
$$
a_\l(\bum, \varphi)  + b(\varphi, p^-)= F(\varphi) \text{ for all }\varphi \in \mathbf{H}^1_{\partial\Om \setminus \Gamma} (\Om), 
$$
where $a_\l(\cdot, \cdot) : \mathbf{H}_{\partial \Om \setminus \Gamma}^1(\Om) \times \mathbf{H}_{\partial \Om \setminus \Gamma}^1(\Om)  \to \R$ is given by
\begin{align}\label{a_lambda}
a_\l(\psi, \varphi) & = \l(\psi, \varphi)_\Om + \l(\mu_\l (\psi|_\Gamma), \mu_\l(\varphi|_\Gamma))_\Op +(\grad \psi, \grad \varphi)_\Om \nonumber \\
 & + (\bU \cdot \grad \mu_\l(\psi|_\Gamma), \mu_\lambda(\varphi|_\Gamma))_\Op \nonumber 
+ \frac12 (\div(\bU)\mu_\l(\psi|_\Gamma), \mu_\l(\varphi|_\Gamma))_\Op  \\
& + (\sigma(\mu_\l(\psi|_\Gamma)), \e(\mu_\l(\varphi|_\Gamma)))_\Op\nonumber
 - (q_\l (\psi|_\Gamma), \div(\mu_\l(\varphi|_\Gamma)))_\Op \nonumber,\\
\end{align}
 $b(\cdot, \cdot) : \mathbf{H}_{\partial \Om \setminus \Gamma}^1(\Om)  \times L^2(\Om) \to \R$ is given by
\begin{align}\label{b}
b(\varphi, p_0) = - (p_0, \div(\varphi))_\Om,  
\end{align}

and $F(\cdot) : \mathbf{H}_{\partial \Om \setminus \Gamma}^1(\Om)  \to \R$ is given by 
\begin{align}\label{F}
F(\varphi) & = (\mathbf{h}, \varphi)_\Om + (\mathbf{f}, \mu_\l(\varphi)|_\Gamma)_\Op - \big[ \l(\tilde\mu(\mathbf{f},g), \mu_\l(\varphi|_\Gamma)))_\Op \\
& + (\bU \cdot \grad \tilde\mu(\mathbf{f},g), \mu_\l(\varphi|_\Gamma))_\Op \nonumber + \frac12(\div(\bU) \tilde\mu(\mathbf{f},g) , \mu_\l(\varphi|_\Gamma))_\Op \\
& + (\sigma(\tilde\mu(\mathbf{f},g)), \e(\mu_\l(\varphi|_\Gamma)))_\Op \nonumber  - (\tilde q(\mathbf{f},g), \div(\mu_\l(\varphi|_\Gamma)))_\Op \big].
\end{align}

We have then that $[\bum, p^-]$ solves 
\beq \label{compact_var_form}
\begin{cases}
a_\l(\bum, \varphi) + b(\varphi, p^-) = F(\varphi) & \text{ for all } \varphi \in \mathbf{H}^1_{\partial \Om \setminus \Gamma} (\Om)\\
-b(\bum, \rho) = 0 &  \text{ for all } \rho \in L^2(\Om)
\end{cases}.
\eeq

With a view towards applying the Babuska-Brezzi Theorem (Theorem \ref{thm:BBT}) to this system, we verify
\begin{itemize}
\item Continuity of $a_\l(\cdot, \cdot)$ and $b(\cdot, \cdot)$:

To show $b(\cdot, \cdot)$ is continuous, let $\varphi \in \mathbf{H}^1_{\partial \Om \setminus \Gamma}(\Om)$ and $\rho \in L^2(\Om)$. Applying the  Cauchy-Schwarz and Poincare Inequalities: 
\begin{align}
b(\varphi, \rho) & = (-\rho, \div(\varphi))_{\Om} \nonumber\\
& \leq ||\rho||_\Om ||\div(\varphi)||_\Om \nonumber \\
& = C ||\rho||_\Om ||\varphi||_{\mathbf{H}^1_{\partial \Om \setminus \Gamma} (\Om)}.
\end{align}

To show $a_\l(\cdot, \cdot)$ is continuous, let $\varphi, \psi \in \mathbf{H}_{\partial \Om \setminus \Gamma}^1(\Om)$. Again applying the Cauchy-Schwarz and Triangle Inequalities to each term in (\ref{a_lambda}) gives,
\begin{align} 
a_\l(\psi, \varphi)  \leq &  \l||\psi||_\Om ||\varphi||_\Om + \l ||\mu_\l(\psi|_\Gamma)||_\Om|| \mu_\l(\varphi|_\Gamma)||_\Om + ||\grad \psi||_\Om ||\grad \varphi||_\Om \nonumber\\
& + ||\sigma(\mu_\l(\psi|_\Gamma))||_\Op||\e(\mu_\l(\varphi|_\Gamma))||_\Op + ||\bU \cdot \grad \mu_\l(\psi|_\Gamma)||_\Op ||\mu_\l(\varphi|_\Gamma)||_\Op \nonumber\\
& \label{acont1} + \frac12 ||\div(\bU) \mu_\l(\psi|_\Gamma)||_\Op ||\mu_\l(\varphi|_\Gamma)||_\Op \nonumber \\
& + ||q_\l(\psi|_\Gamma)||_\Op||\div(\mu_\l(\varphi|_\Gamma))||_\Op.
\end{align}

Now recall that $\bU \in \mathbf{H}^2(\Op)$ so by Sobolev embedding with dimension $n = 2$, we have $\bU \in C^1\left(\overline{\Op}\right)$. Thus, 
\beq \label{Uinfnorm} ||\bU\cdot \grad \mu_\l(\psi|_\Gamma)||_\Op \leq ||\bU||_{\mathbf{L}^\infty(\Op)} ||  \grad \mu_\l(\psi|_\Gamma)||_\Op \eeq and 
\beq \label{gradUinfnorm}  ||\div(\bU) \mu_\l(\psi|_\Gamma)||_\Op \leq ||\div(\bU)||_{\mathbf{L}^\infty(\Op)} ||\mu_\l(\psi|_\Gamma)||_\Op,
\eeq

with $||\bU||_{\mathbf{L}^\infty(\Op)} < \infty$ and $||\div(\bU)||_{\mathbf{L}^\infty(\Op)} < \infty$.

And since $D_\l(\cdot) = \bbm \mu_\l(\cdot) \\ q_\l(\cdot) \ebm$, as defined in (\ref{Dlambda}), is a well-posed mapping, giving continuous dependence on data, and using Sobolev trace regularity, we have 
\beq \label{ctnsdep} ||\mu_\l (\varphi|_\Gamma)||_\Op + ||q_\l(\varphi|_\Gamma)||_\Op \leq C||\varphi||_\Om.\eeq

With $\psi, \varphi \in \mathbf{H}^1_{\partial \Om\setminus \Gamma}(\Om)$, combining (\ref{acont1}), (\ref{Uinfnorm}), (\ref{gradUinfnorm}), (\ref{ctnsdep}), and Poincare's Inequality 
gives
\begin{align}
a_\l(\psi, \varphi)  \leq  \, \, \, &C \big[  \l||\psi||_{\mathbf{H}^1_{\partial \Om \setminus \Gamma} (\Om)} ||\varphi||_{\mathbf{H}^1_{\partial \Om \setminus \Gamma}(\Om)} + 
\l ||\psi||_{\mathbf{H}^1_{\partial \Om \setminus \Gamma}(\Om)}|| \varphi||_{\mathbf{H}^1_{\partial \Om \setminus \Gamma}(\Om)} \nonumber \\
&+ ||\psi||_{\mathbf{H}^1_{\partial \Om \setminus \Gamma}(\Om)} ||\varphi||_{\mathbf{H}^1_{\partial \Om \setminus \Gamma}(\Om)} 
 + ||\sigma(\mu_\l(\psi|_\Gamma))||_\Op||\e(\mu_\l(\varphi|_\Gamma))||_\Op \nonumber \\
&+ ||\bU||_{\mathbf{L}^\infty(\Op)} || \psi||_{\mathbf{H}^1_{\partial \Om \setminus \Gamma}(\Om)} ||\varphi||_{\mathbf{H}^1_{\partial \Om \setminus \Gamma}(\Om)} \nonumber\\
& \label{acont1} + \frac12 ||\div(\bU)||_{\mathbf{L}^\infty(\Op)}||\psi||_{\mathbf{H}^1_{\partial \Om \setminus \Gamma}(\Om)} ||\varphi||_{\mathbf{H}^1_{\partial \Om \setminus \Gamma}(\Om)} \nonumber \\
& + ||\psi||_{\mathbf{H}^1_{\partial \Om \setminus \Gamma}(\Om)}||\varphi||_{\mathbf{H}^1_{\partial \Om \setminus \Gamma}(\Om)}\big].
\end{align}

Further, note that, given the definitions of $\sigma(\cdot)$ and $\epsilon(\cdot)$ in terms of first partial derivatives, 
\begin{align}
||\sigma(\mu_\l(\psi|_\Gamma))||_\Op||\e(\mu_\l(\varphi|_\Gamma))||_\Op&  \leq  \tilde C ||\mu_\l(\psi|_\Gamma)||_{\mathbf{H}^1(\Op)} ||\mu_\l(\varphi|_\Gamma))||_{\mathbf{H}^1(\Op)} \nonumber \\
& \leq C^*||\psi||_\Om ||\varphi||_\Om  \nonumber \\
& \leq C^*||\psi||_{\mathbf{H}^1(\Om)} ||\varphi||_{\mathbf{H}^1(\Om)},
\end{align}
where the second line follows from the wellposedness of $\mu_\lambda$ and $q_\l$, as well as Sobolev trace regularity. 
Putting all of this together gives
\beq 
a_\l(\psi, \varphi) \leq C_U ||\psi||_{\mathbf{H}^1_{\partial \Om \setminus \Gamma}(\Om)} ||\varphi||_{\mathbf{H}^1_{\partial \Om \setminus \Gamma}(\Om)},
\eeq 
which gives continuity.

\item To show $a_\l(\cdot, \cdot)$ is $\mathbf{H}^1_{\partial \Om \setminus \Gamma}(\Om)$-elliptic, first note that by the second equation in  (\ref{Dlambda}), we have for all $\varphi \in \mathbf{H}^1_{\partial \Om \setminus \Gamma}$,
$$\div(\mu_\l(\varphi|_\Gamma)) = -\l q_\l(\varphi|_\Gamma) - \bU \cdot \grad q_\l(\varphi|_\Gamma) - \frac12 \div(\bU) q_\l(\varphi|_\Gamma) \, \, \text{ in } \Op.$$

Thus for $\varphi \in \mathbf{H}^1_{\partial \Om \setminus \Gamma}(\Om)$, 
\begin{align}  \label{ellipticity_of_alambda}
a_\l(\varphi, \varphi)  = & \l ||\varphi||^2_\Om + \l ||\mu_\l(\varphi|_\Gamma)||^2_\Op + ||\grad \varphi||^2_\Om + (\sigma(\mu_\l(\varphi|_\Gamma)), \e(\mu_\l(\varphi|_\Gamma)))_\Op \nonumber \\
& + (\bU \cdot \grad \mu_\l(\varphi|_\Gamma), \mu_\l(\varphi|_\Gamma))_\Op  + \frac12 (\div(\bU) \mu_\l(\varphi|_\Gamma), \mu_\l(\varphi|_\Gamma))_\Op \nonumber \\
& - (q_\l(\varphi|_\Gamma), \div(\mu_\l(\varphi|_\Gamma)))_\Op.
\end{align}

 Note that, as in (\ref{a_elliptic}), 
 
\begin{align} \label{mu_pos_semi_def}
(\sigma(\mu_\l(\varphi|_\Gamma)), \e(\mu_\l(\varphi|_\Gamma)))_\Op & = 2 \nu ||\epsilon(\mu_\l(\varphi|_\Gamma))||^2_{\Op} + \lambda( \div(\mu_\l(\varphi|_\Gamma)))^2 \nonumber \\
&  \geq 0. 
\end{align}

 Also, by (\ref{2Re}), 
 \beq \label{mu_2Re}
 (\bU \cdot \grad \mu_\l(\varphi|_\Gamma), \mu_\l(\varphi|_\Gamma))_\Op  + \frac12 (\div(\bU) \mu_\l(\varphi|_\Gamma), \mu_\l(\varphi|_\Gamma))_\Op = 0.
 \eeq 

Finally, note that, with (\ref{Dlambda}) in mind, 
$$
\div(\mu_\l(\varphi|_\Gamma)) = - \l q_\l(\varphi|_\Gamma) - \bU \cdot \grad q_\l(\varphi|_\Gamma) - \frac12 \div(\bU) q_\l(\varphi|_\Gamma),
$$
and so 
\begin{align}\label{q_div}
-(q_\l(\varphi|_\Gamma), \div(\mu_\l(\varphi|_\Gamma)))_\Op & = (q_\l(\varphi|_\Gamma),  \l q_\l(\varphi|_\Gamma) + \bU \cdot \grad q_\l(\varphi|_\Gamma) \nonumber \\
& + \frac12 \div(\bU) q_\l(\varphi|_\Gamma))_\Op \nonumber \\
& = \l||q_\l(\varphi|_\Gamma)||_{\Op}^2 + 0\\
& \geq 0, 
\end{align}
where the  $0$ in the second to last line comes from (\ref{2Re}) again.

Thus, using Poincar\`e's Inequality on the first and third terms of (\ref{ellipticity_of_alambda}), and (\ref{mu_pos_semi_def}) and (\ref{q_div}) on  the others, 

\begin{align} 
a_\l(\varphi, \varphi) & \geq \l \cdot C ||\varphi||^2_{\mathbf{H}^1_{\partial \Om \setminus \Gamma}(\Om)} +0 + \tilde C||\varphi||^2_{\mathbf{H}^1_{\partial \Om \setminus \Gamma}(\Om)} + 0  + 0 + 0 \nonumber \\
& \geq \alpha ||\varphi||_{\mathbf{H}^1_{\partial \Om \setminus \Gamma}(\Om)}^2,
\end{align}
for some constant $\alpha > 0$, which establishes ellipticity.

\item To show the so-called Inf-Sup condition, we make note of the following lemma from \cite{Galdi}: 
\begin{lemma}
For $\Omega \subset \mathbb{R}^n$ that is bounded, open, and with Lipshitz boundary $\partial \Omega $, there exists some $\delta > 0$ and $\mu \in [C^\infty(\bar{\Omega})]^n$ such that $\mu \cdot \nu \geq \delta$ a.e. on $\partial \Omega$.
\end{lemma} 

Thus, for our domain $\Om$, such a $\delta > 0$ and $\mu \in \mathbf{C}^\infty(\bar{\Om})$ exist with $\mu \cdot \nu \geq \delta$ a.e. on $\partial \Om$. With this $\mu$ in hand, let $\omega \in \mathbf{H}^1_{\partial \Om \setminus \Gamma} (\Om)$  be a solution to 
\beq 
\begin{cases}
\div(\omega) = - \eta \la \mu, \vec{\nu} \ra_\Gamma & \text{ in } \Om, \\
\omega|_{\partial \Om \setminus \Gamma} = 0 & \text{ on } \partial \Om \setminus \Gamma, \\
\omega|_{\Gamma} = \left( \int_{\Om} \eta \, d\Om \right) \mu(x) & \text{ on } \Gamma
\end{cases}
\eeq
for any $\eta \in L^2(\Om)$. Note that this solution, $\omega$, exists by \cite{Galdi} with \\ $||\grad \omega||_\Om \leq C ||\eta||_\Om$.
With these in hand, we now consider 
\begin{align}
\sup_{\varphi \in \mathbf{H}^1_{\partial \Om \setminus \Gamma}(\Om)} \frac{b(\varphi, \eta)}{||\varphi||_{\mathbf{H}^1_{\partial \Om \setminus \Gamma}(\Om)}} & = \sup_{\varphi \in \mathbf{H}^1_{\partial \Om \setminus \Gamma}} \frac{b(\varphi, \eta)}{||\grad \varphi||_{\mathbf{L}^2(\Om)}} \nonumber \\
& = \sup_{ \varphi \in \mathbf{H}^1_{\partial \Om \setminus \Gamma}(\Om)} \frac{- \int \eta \div(\varphi) \, d\Om}{||\grad\varphi||_\Om}  \nonumber \\
& \geq \frac{ - \int \eta \div(\omega) \, d\Om}{||\grad \omega||_\Om} \nonumber \\
& = \frac{\int \eta^2 \la \mu, \nu \ra_{\Gamma} \, d\Om}{||\grad \omega||_\Om} \nonumber \\
& \geq \frac{\delta \cdot \text{meas}(\Gamma) ||\eta||^2_\Om}{||\grad \omega||_\Om} \nonumber \\
& \geq \frac{\delta \cdot \text{meas}(\Gamma) ||\eta||_\Om \left( \frac1C ||\grad \omega||_\Om\right)}{||\grad \omega||_\Om} \nonumber \\
& = \frac1C \,\delta\,  \text{meas}( \Gamma) ||\eta||_\Om.
\end{align}
Thus, since $\eta \in L^2(\Om)$ was arbitrary, we have
\beq
\sup_{\varphi \in \mathbf{H}^1_{\partial \Om \setminus \Gamma}(\Om)} \frac{b(\varphi, \eta)}{||\varphi||_{\partial \Om \setminus \Gamma}} \geq \beta ||\eta||_\Om,
\eeq
with $\beta =  \frac1C \,\delta\,  \text{meas}( \Gamma)$, and so the Inf-Sup condition is satisfied. 

\end{itemize}
In view of the results demonstrated in this subsection, the Babuska-Brezzi Theorem applies to (\ref{compact_var_form}), giving the existence and uniqueness of solutions $[\bum, p^-]$. From there, we  recover \begin{align}\label{up,pp} \bbm \bup \\ p^+ \ebm & = D_\l(\bum|_\Gamma) + \bA^{-1}_\l([\mathbf{f}, g]^T)\nonumber.\\
\end{align}

It remains to show that $[\bup, p^+, \bum] \in \mathcal{D}(\mathcal{A})$. To this end, note that 
\begin{enumerate}

\item[(A.1)] Since $[\bum, p^-] $ solves (\ref{compact_var_form}), it follows that $\bum \in \mathbf{H}^1_{\partial \Om \setminus \Gamma}(\Om)$, $\div(\bum) = 0$ a.e. in $\Om$, and $[\bum \cdot \vec{\nu}]_{\partial \Om \setminus \Gamma} = 0$. Furthermore, $D_\l: \mathbf{H}^{1/2} (\Gamma) \to \mathbf{H}^1_{\partial \Op \setminus \Gamma}(\Op) \times L^2(\Op)$ and $\mathbb{A}_\l: \mathbf{H}^1_0(\Op) \times L^2(\Op) \to \mathbf{L}^2(\Op) \times L^2(\Op)$, so $[\bup, p^+] \in \mathbf{H}^1_{\partial \Op \setminus \Gamma}(\Op) \times L^2(\Op)$. Thus, we have $[\bup, p^+, \bum] \in \mathcal{H}$  with 
$$ [\bup, p^+, \bum]  \in \mathbf{H}^1_{\partial \Omega^+ \setminus \Gamma} (\Omega^+) \times L^2(\Omega^+) \times [\mathbf{H}^1_{\partial \Omega^- \setminus \Gamma}(\Omega^-)].$$

\item[(A.2)] Identifying $p^+ = q_\l (\bum|_\Gamma) + \tilde q(\mathbf{f}, g)$ and given the $\mathbf{L}^2$-equality in (\ref{q_lambda_var_form}), (\ref{resolve_qg}), and (\ref{resolve_qdivmu}), we have that $\bU \cdot \grad p^+ = g - (\l p^+ + \div(\bup) + \frac12 \div(\bU) p^+)$ a.e.  with $g, p^+ \in L^2(\Op)$.
Further, $\bup \in \mathbf{H}^1_{\partial \Op \setminus \Gamma}(\Op)$ so $\div(\bup) \in \mathbf{L}^2(\Op)$. 
Finally, note that Sobolev embedding gives $\bU \in \mathbf{C}^1(\overline{\Op})$, so $\div(\bU) \in C(\overline{\Op})$ and $||\div (\bU) ||_\infty < \infty$. 
Combining all of these gives $\bU \cdot \grad p^+ \in L^2(\Op)$. 

\item[(A.3)] Identifying $\bup = \mu_\l(\bum|_\Gamma) + \tilde \mu(\mathbf{f},g)$ and given the a.e. equality in (\ref{mu_lambda_var_form}) and (\ref{tilde_mu_var_form}), we have that $$\l \bup + \bU \cdot \grad \bup - \div \, \sigma( \bup) + \frac12 \div(\bU) \bup + \grad p^+ = \mathbf{f}.$$ 
Thus, $- \div \, \sigma(\bup) + \grad p^+ = \mathbf{f} - \lambda \bup - \bU \cdot \grad \bup$, with $\mathbf{f} \in \mathbf{L}^2(\Op)$.
Further, $\bup \in \mathbf{H}^1_{\partial \Op \setminus \Gamma}$ and, as before, $\bU \in \mathbf{C}^1(\overline{\Op})$, so $\bU \cdot \grad \bup \in \mathbf{L}^2(\Op)$. 
Combining these gives $- \div \, \sigma(\bup) + \grad p^+ \in \mathbf{L}^2(\Op)$. 

\item[(A.4)] Recall that $\bup = \mu_\l(\bum|_\Gamma) + \tilde \mu(\mathbf{f}, g)$. Thus, given the BVP's given in (\ref{Dlambda}) and (\ref{psystem}), we have 

$$
\bup|_{\Gamma} = \bum|_\Gamma + 0.
$$

\item[(A.5)] Finally, since $[\bum, p^-]$ solve (\ref{compact_var_form}),  we have $\bum \in \mathbf{H}^1_{\partial \Om \setminus \Gamma}(\Om)$,  $p^-\in L^2(\Om)$, and $\div(\bum)  = 0$. 
Furthermore, from (\ref{Lap_um_plus_gradpm}),
$$ 
(\l \bum , \varphi)_{\Omega^-} - (\Lap \bum, \varphi)_{\Om} + (\grad p^-, \varphi)_\Om  = (\mathbf{h}, \varphi)_\Om, 
$$ for all $\varphi \in \mathbf{H}^1_{\partial \Om \setminus \Gamma}(\Om)$. In particular, this holds for all $\varphi \in \mathcal{D}(\Om)$, which gives 
\beq
\lambda \bum - \Lap \bum + \grad p^- = \mathbf{h} \label{um_equn_inL2}
\eeq
in the $L^2$-sense (i.e. almost everywhere). 
Since $\mathbf{h} \in \mathbf{L}^2(\Om) $, it follows then that $-\Lap \bum + \grad p^- = \mathbf{h} - \lambda \bum \in \mathbf{L}^2(\Om)$. 
Additionally, note that 
$$\div(-\Lap \bum + \grad p^-) = - \Lap (\div (\bum)) + \Lap p^- = 0 + 0,$$
 as $\div(\bum) = 0$ and $\Lap p^- = 0$ by (\ref{bvp_for_p}). Finally, since $p^-$ is a solution to (\ref{bvp_for_p}), we have that 
 $\frac{\partial p^-}{\partial \vec{\nu}} = \Lap \bum \cdot \nu$ on $\partial \Om \setminus \Gamma$. Rearranging, we have $$(-\Lap \bum + \grad p^-) \cdot \vec{\nu} = 0$$ on $\partial \Om \setminus \Gamma$. Thus, (A.5a) is established, giving the additional regularity prescribed by Lemma 2. 

Furthermore, from (\ref{compact_var_form}), we also have 
\begin{align}
& \lambda (\bum, \varphi)_\Om + \l (\mu_\l(\bum|_\Gamma), \mu_\l(\varphi|_\Gamma))_\Op + (\grad \bum, \grad \varphi)_\Om  \\
& + (\bU \cdot \grad \mu_\l(\bum|_\Gamma), \mu_\l(\varphi|_\Gamma))_{\Op}\nonumber   + \frac12 (\div(\bU) \mu_\l(\bum|_\Gamma), \mu_\l (\varphi|_\Gamma))_\Op \\
& + (\sigma(\mu_\l(\bum|_\Gamma)), \epsilon(\mu_\l(\varphi|_\Gamma)))_\Op \nonumber - (q_\l(\bum|_\Gamma), \div(\mu_\l(\varphi|_\Gamma)))_\Op- (p^-, \div \varphi)_\Om \\
& = (\mathbf{h} , \varphi)_\Om + (\mathbf{f}, \mu_\l(\varphi|_\Gamma))_\Op \nonumber \\
&- [\l (\tilde \mu, \mu_\l(\varphi|_\Gamma))_\Op + ( \bU \cdot \grad \tilde \mu, \mu_\l(\varphi|_\Gamma))_\Op + \frac12 (\div(\bU) \tilde \mu, \mu_\l(\varphi|_\Gamma))_\Op  \nonumber \\
& + (\sigma(\tilde \mu), \epsilon(\mu_\l(\varphi|_\Gamma)))_\Op - (\tilde q, \div \, \mu_\l(\varphi|_\Gamma))_\Op]
\end{align}
for all $\varphi \in \mathbf{H}^1_{\partial \Om \setminus \Gamma}(\Om)$. 
Identifying $\bup = \mu_\l(\bum|_\Gamma) + \tilde \mu(\mathbf{f},g)$ as in (\ref{up,pp}), this becomes
\begin{align}
& \l(\bum, \varphi)_\Om + (\grad \bum, \grad \varphi)_{\Om} - (p^-, \div(\varphi))_\Om + [ \l(\bup, \mu_\l(\varphi|_\Gamma))_\Op\nonumber  \\
& (\bU \cdot \grad \bup, \mu_\l(\varphi|_\Gamma))_\Op + \frac12(\div(\bU) \bup, \mu_\l(\varphi|_\Gamma))_\Op \nonumber\\
&+ (\sigma(\bup), \e(\mu_\l(\varphi|_\Gamma)))_\Op - (p^+, \div(\mu_\l(\varphi|_\Gamma)))_\Op - (\mathbf{f}, \mu_\l(\varphi|_\Gamma))_{\Op} ]  = (\mathbf{h},\varphi)_\Om.
\end{align}

Further, working backwards using (\ref{sigma_on_bound}), we then have 
\begin{align}
&  \l (\bum , \varphi)_\Om + ( \grad \bum, \grad \varphi)_\Om - (p^-, \div \varphi)_\Om + \la \sigma(\bup)\vec{\nu} - p^+ \vec{\nu}, \varphi \ra_{\Gamma} = (\mathbf{h}, \varphi)_\Om. 
\end{align}
Using Green's Theorem on the second and third terms gives
\begin{align}
&(\l\bum  - \Lap \bum + \grad p^- - \mathbf{h}, \varphi)_\Om + \left< \frac{\partial \bum}{\partial \vec{\nu}} \vec{\nu} - p^- \vec{\nu}, \varphi \right>_\Gamma \nonumber \\
& \hspace{.5in} +  \la \sigma(\bup)\vec{\nu} - p^+ \vec{\nu}, \varphi \ra_{\Gamma} = 0
\end{align}
for all $\varphi \in \mathbf{H}^1_{\partial \Om \setminus \Gamma}(\Om)$. Using (\ref{um_equn_inL2}), this gives 
\beq
\left< \frac{\partial \bum}{\partial \vec{\nu}} \vec{\nu} - p^- \vec{\nu}, \varphi \right>_\Gamma +  \la \sigma(\bup)\vec{\nu} - p^+ \vec{\nu}, \varphi \ra_{\Gamma} = 0\text{ for all } \varphi \in \mathbf{H}^1_{\partial \Om \setminus \Gamma}(\Om).
\eeq

Using surjectivity of the Sobolev trace map, this now gives 
\beq 
\sigma(\bup) \vec{\nu} - p^+ \vec{\nu} = \frac{\partial \bum}{\partial \vec{\nu}} - p^- \vec{\nu} \text{ on } \Gamma \text{ in } \mathbf{H}^{-1/2}(\Gamma),
\eeq
which establishes (A.5b).

\end{enumerate}

Thus, we indeed have $[\bup, p^+, \bum] \in \mathcal{D}(\mathcal{A})$.

Since we have shown $\mathcal{A}$ is maximal dissipative, the Lumer-Phillips Theorem applies, giving a $C_0$-semigroup of contractions on $ \mathcal{H}$, establishing Theorem 1(i). Additionally, the identifications in (\ref{compact_var_form}) and (\ref{up,pp}) establish Theorem 1(ii).

\section{Numerical Analysis of the Fluid-Fluid Dynamics}

The objective of this section is to demonstrate how the maximality argument given in Section \ref{sec:Maximality} can be utilized to approximate solutions to the fluid-fluid interaction PDE system under present consideration. In particular, the numerical method outlined here solves the static problem given by (\ref{varequs}) - (\ref{resolventequn}), but this approach can be modified to solve the time dependent problem in the same way as in \cite{Avalos&Dvorak}, via the framework laid out in \cite{Pazy}. That is, after obtaining the solution of the static problem, the exponential formula for the semigroup is used to generate time-dependent solutions, i.e. 
  \beq
  \bbm \bup(t) \\ p^+(t) \\ \bum(t) \ebm = e^{\mathcal{A}t} \bbm \mathbf{u}_0^+ \\ p^+_0 \\ \mathbf{u}_0^- \ebm 
  = \lim_{n \to \infty} \left( I - \frac{t}{n}\mathcal{A} \right)^{-n} \bbm \mathbf{u}_0^+ \\ p^+_0 \\ \mathbf{u}_0^- \ebm 
 \text{ for } \bbm \mathbf{u}_0^+ \\ p^+_0 \\ \mathbf{u}_0^- \ebm \in \mathcal{H}.
  \eeq
Setting $\lambda = \frac{n}{t}$ in this equation gives
\beq
(\lambda I - \mathcal{A})^n \bbm \bup(t) \\ p^+(t) \\ \bum(t) \ebm = \lambda^n \bbm \mathbf{u}_0^+ \\ p^+_0 \\ \mathbf{u}_0^- \ebm ,
\eeq  
which can be solved using the scheme derived from the maximality argument. By choosing $n$ large enough to approximate the exponential semigroup operator, one can then recover the solution to the time dependent system at any time $t \geq 0$. 

Here we outline a numerical implementation of obtaining approximate solutions using the finite element method (FEM) and associated convergence results with respect to the discretization parameter, $h$. We conclude by providing a specific ``test problem" as a numerical example, with associated error analysis.

\subsection{Finite element formulation} \label{sec:approx_sols}
The finite element method is a numerical implementation of the Ritz-Galerkin method using certain basis functions defined on a mesh on the domain. In this case, both domains, $\Op$ and $\Om$, are divided into triangles, with basis functions associated with points, or ``nodes", on the mesh. In this finite dimensional setting, the system is then set forth as a matrix-vector equation whose solution gives the coefficients of each basis function (see \cite{AB}). In this section, we show that the discrete FEM formulation of (\ref{varequs}) - (\ref{resolventequn}) is wellposed. 

In what follows, the domains $\Om$ and $\Op$ will be taken as before, that is, two rectangles sharing a common interface, $\Gamma$. Now given $h$, a small positive discretization parameter, let $
\{e_\ell\}_{\ell = 1}^{N_h}$ be a FEM ``triangulation" of $\Om$ and $ \{ \tilde e_\ell\}_{\ell = 1}^{\tilde N_h}$ be a FEM ``triangulation" of $\Op$, where each element $e_\ell$ and $\tilde e_\ell$ is a triangle and (among other properties) $\cup_{\ell = 1}^{N_h} e_\ell = \Om$ and $\cup_{\ell = 1}^{\tilde N_h} \tilde e_\ell = \Op$.

A few things to note:
\begin{enumerate}
\item[(A)] Relative to the ``triangulation" of $\Om$, $V_h^-$ will denote the classic $\mathbf{H}^1$-conforming FEM finite dimensional subspace for the fluid $\bum$ variable such that 

\begin{align} \label{Vhspace}
\mu_h \in V^-_h \Rightarrow \left. \mu_h\right|_{e_\ell} \in [\mathbb{P}_2]^2, \, \, 
&V_h^- \subset \mathbf{H}_0^1(\Om), \, \,  
V_h^- \not\subset \mathbf{H}^2(\Om); \, \,  \\
&V_h^- \subset \mathbf{C}(\overline{\Om}), \, \, V_h^- \not\subset \mathbf{C}^1(\overline{\Om}). \nonumber
\end{align}
(See \cite{AB}.)

To handle inhomogeneity, specify the set 
\begin{align} 
\tilde{V}_h^- = &\left\{\mu_h + \gamma_0^+(\xi) \in \mathbf{H}^1(\Om): \right.\\
 & \left. \mu_h \in V_h^-, \, \,  \gamma_0^+(\xi) = \begin{cases}  \mathbf{0} \text{ on } \partial\Om \setminus \Gamma\\
\xi \text{ on } \Gamma
\end{cases}\right.  \text{ for } \xi \in \mathbf{H}^1(\Om)\}.
\end{align}

\item[(B)] Similarly, relative to the ``triangulation" of $\Op$, $V_h^+$ will denote the classic $\mathbf{H}^1$-conforming FEM finite dimensional subspace for the fluid $\bup$ variable such that 

\begin{align} \label{Vh+space}
\mu_h \in V^+_h \Rightarrow \left. \mu_h\right|_{\tilde e_\ell} \in [\mathbb{P}_2]^2, \, \, 
&V_h^+ \subset \mathbf{H}_0^1(\Op), \, \,  
V_h^+ \not\subset \mathbf{H}^2(\Op); \, \,  \\
&V_h^+ \subset \mathbf{C}(\overline{\Op}), \, \, V_h^+ \not\subset \mathbf{C}^1(\overline{\Op}). \nonumber
\end{align}
(See \cite{AB}.)

To handle inhomogeneity, specify the set 
\begin{align}
\tilde{V}_h^+ = & \left\{  \mu_h + \gamma_0^+(\xi) \in \mathbf{H}^1(\Op): \right.\\
& \left. \mu_h \in V_h^+, \, \,  \gamma_0^+(\xi) = \begin{cases}  \textbf{0} \text{ on } \partial\Op \setminus \Gamma\\
\xi \text{ on } \Gamma 
\end{cases} \right. \text{ for } \xi \in \mathbf{H}^1(\Op) \} .
\end{align}

\item[(C)] In addition, $\Pi^-_h$ will denote the $L^2$-FEM finite dimensional subspace for the pressure variable, $p^-$, defined by 
\beq \label{Pi-space}
\Pi^-_h = \{ q_h \in L^2(\Om)\cap C(\overline \Om): \forall \ell = 1, \ldots, N_h; , \left. q_h\right|_{e_\ell} \in \mathbb{P}^1\}
\eeq

(See \cite{AB}.)

\item[(D)] Similarly, $\Pi^+_h$ will denote the $L^2$-FEM finite dimensional subspace for the pressure variable, $p^+$, defined by 
\beq \label{Pi+space}
\Pi^+_h = \{ q_h \in L^2(\Op)\cap C(\overline \Op): \forall \ell = 1, \ldots, \tilde N_h; , \left. q_h\right|_{\tilde e_\ell} \in \mathbb{P}^1\}
\eeq
(See \cite{AB}.)
\end{enumerate}

For the spaces $V^-_h$, $V^+_h$, $\Pi^-_h$, and $\Pi^+_h$ described above, we will have need of the following discrete estimates relative to mesh parameter $h$:

\begin{enumerate}
\item[(A')] In regard to the $\mathbf{H}^1(\Om)$-conforming FEM space $V_h^-$ in (\ref{Vhspace}) we have the following estimate: For $\mu \in \mathbf{H}^2(\Om) \cap \mathbf{H}_0^1(\Om)$,
\beq \label{um_bound}
\min_{\mu_h \in V_h^-} || \mu - \mu_h||_{\mathbf{H}^1_0(\Om)} \leq C h |\mu|_{2,\Om}.
\eeq
(See Theorem 5.6, p. 224, of \cite{AB}.)

\item[(B')] In regard to the $\mathbf{H}^1(\Op)$-conforming FEM space $V_h^+$ in (\ref{Vh+space}) we have the following estimate: For $\mu \in \mathbf{H}^2(\Op) \cap \mathbf{H}_0^1(\Op)$,
\beq \label{up_bound}
\min_{\mu_h \in V_h^+} || \mu - \mu_h||_{\mathbf{H}^1_0(\Op)} \leq C h |\mu|_{2,\Op}.
\eeq
(See Theorem 5.6, p. 224, of \cite{AB}.)

\item[(C')] Similarly, in regard to the finite dimensional space $\Pi^-_h$, we have the discrete estimate: For $q \in H^1(\Om)$, 

\beq \label{pm_bound}
\min_{q_h \in \Pi^-_h} || q - q_h||_{L^2(\Om)} \leq C h ||q||_{H^1(\Om)}.
\eeq
(See e.g., Corollary 1.128, pg. 70, of \cite{Guermond}.)
\item[(D')] Finally, in regard to the finite dimensional space $\Pi^+_h$, we have the discrete estimate: For $q \in H^1(\Op)$, 

\beq \label{pp_bound}
\min_{q_h \in \Pi^+_h} || q - q_h||_{L^2(\Op)} \leq C h ||q||_{H^1(\Op)}.
\eeq
(See e.g., Corollary 1.128, pg. 70, of \cite{Guermond}.)

\end{enumerate}

The goal here is to find a finite dimensional approximation $[\mathbf{u}_h^-, p^-_h] \in V_h^- \times \Pi^-_h$ to the solution $[\bum, p^-]$ of (\ref{compact_var_form}), as well as an approximation $[\mathbf{u}^+_h, p^+_h] \in V_h^+ \times \Pi_h^+$ of the fluid and pressure in $\Op$, respectively. 
These  particular FEM subspaces are chosen with a view to satisfy the discrete Babuska-Brezzi condition relative to a mixed variational formulation, a formulation which is entirely analogous to (\ref{compact_var_form}) for the static fluid-fluid PDE system (\ref{varequs}) - (\ref{end_resolvent}).

In line with the maximality argument of Section \ref{sec:Maximality}, the initial task in the present finite dimensional setting is to numerically resolve the fluid and pressure variables in $\Om$ in the PDE system (\ref{varequs}). Namely, with the bilinear and linear functionals $a_\lambda(\cdot, \cdot)$, $b(\cdot, \cdot)$, and $F(\cdot)$ of (\ref{a_lambda}) - (\ref{F}), the present discrete problem is to find 
$$
[\mathbf{u}_h^-, p_h^-] \in V_h^- \times \Pi_h^- 
$$
which solve 
\beq \label{discrete_varequs}
\begin{cases}
a_\lambda(\mathbf{u}_h^-, \mathbf{\varphi}_h^-) + b(\mathbf{\varphi}_h^-, p_h^-) & = F(\mathbf{\varphi}_h^-) \, \, \, \,  \,  \forall \, \,  \mathbf{\varphi}_h^- \in V_h^-,\\
\hspace{.85in} -b(\mathbf{u}_h^-, \mathbf{\psi}_h^-) & = 0 \, \, \, \ \, \, \hspace{.33in} \forall \, \, \mathbf{\psi}_h^- \in \Pi_h^-.
\end{cases}
\eeq

As $V_h^-$ and $\Pi_h^-$ are closed subspaces of  $\mathbf{H}_{\partial \Om \setminus \Gamma}^1(\Om)$ and $L^2(\Om)$, the linearity, continuity, and ellipticity arguments of previous sections also extend to (\ref{discrete_varequs}). Additionally, the necessary (discrete) inf-sup condition 
\beq
\sup_{\varphi_h \in V_h^-} \frac{b(\varphi_h^-, \eta)}{||\varphi_h^-||_{\mathbf{H}_{\partial\Om \setminus \Gamma}(\Om)}} \geq \beta_h ||\eta||_{L^2(\Om)} \, \, \, \, \forall \, \, \eta \in L^2(\Om)
\eeq
is satisfied uniformly in parameter $h$ (at least for $h$ small enough). The discrete inf-sup condition  has historically been nontrivial to show in uncoupled fluid flows, but see Lemma 4.23 in \cite{Guermond} or Theorem 3.1 in \cite{Avalos&Toundykov} for a proof. 

Thus, one has the unique solvability of this discrete problem (\ref{discrete_varequs}) via the Babuska-Brezzi Theorem, giving the desired  $[\mathbf{u}_h^-, p_h^-] \in V_h^- \times \Pi_h^- $.

From these, as mentioned above, we immediately recover the fluid and pressure approximations in $\Op$, $\mathbf{u}_h^+$ and $p^+_h$, from the discrete solution pair $[\mathbf{u}_h^-, p_h^-] \in V_h^- \times \Pi_h^- $ via
\beq \label{mapping_sols}
 \bbm \mathbf{u}_h^+ \\
p^+_h \ebm 
= [D_{\lambda,h}(\mathbf{u}_h^-|_\Gamma) + \mathbb{A}^{-1}_{\l,h}(\mathbf{f},g)]
\eeq
for $f \in \mathbf{L}^2(\Op), \, \, g \in L^2(\Op)$, and  $D_{\lambda, h}(\cdot)$ as in (\ref{Dlambda}) and $\mathbb{A}_{\l,h}$ as in (\ref{AGWStep1}). 

We note here that these maps, $D_{\l,h}: \mathbf{H}^{1/2} (\Gamma) \to [V_h^+ \cap \, \mathbf{H}_{\partial \Op \setminus \Gamma}^1(\Op)] \times \Pi_h^+$ and $\mathbb{A}_{\l,h}^{-1}: \mathbf{L}^2(\Op) \times L^2(\Op) \to [V^+_h \cap \,  \mathbf{H}_0^1(\Op)] \times \Pi_h^+ $, are entirely analogous to those defined in (\ref{Dlambda}) and (\ref{AGWStep1}) and are wellposed in the discrete setting. Indeed, since $V_h^+$ is a closed subspace of $\mathbf{H}^1_{\partial \Op \setminus \Gamma}(\Op)$ and $\mathbf{H}^1_{0}(\Op)$, and $\Pi^+_h$ is a closed subspaces of $L^2(\Op)$, the same arguments as in Lemma \ref{Dlemma} and Lemma \ref{A_lemma} apply. 

\subsection{Error estimates for the finite element problem}
We now turn to determining error bounds for approximate solutions obtained in Section \ref{sec:approx_sols}, in particular, establishing convergence rates involving the mesh parameter, $h$.  In what follows, we will make use of the following result (Lemma 2.44) from \cite{Guermond}, which is not stated here in its full generality.

\begin{lemma} \label{lem:Guermond}
With reference to the quantities in Theorem \ref{thm:BBT} above, let $\Sigma_h$ be a subspace of $\Sigma$, and let $M_h$ be a subspace of $M$. Suppose further that the bilinear form $a: \Sigma \times \Sigma \to \R$ is $\Sigma$-elliptic; that is, $\exists \alpha > 0$ such that 
\beq	\label{coercive}
a(\sigma, \sigma) \geq \alpha ||\sigma||_{\Sigma}^2 \, \, \, \forall \, \sigma \in \Sigma.
\eeq 
Also, assume that the following ``discrete inf-sup" condition is satisfied: $\exists \, \beta_h > 0$ such that 
\beq \label{discrete_infsup}
\inf_{q_h \in M_h} \sup_{\tau_h \in \Sigma_h} \frac{b(\tau_h, q_h)}{||\tau_h||_\Sigma ||q_h||_M} \geq \beta_h,
\eeq
where $\beta_h >0 $ may depend on subspaces $\Sigma_h$ and $M_h$. Moreover, let $(\sigma_h, p_h) \in \Sigma_h \times M_h$ solve the following (approximating) variational problem:
\beq
\begin{cases}
a(\sigma_h, \tau_h) + b(\tau_h,p_h) =  (\kappa, \tau_h) & \forall \tau_h \in \Sigma_h;\\
b(\sigma_h, q_h) = (l, q_h), & \forall q_h \in M_h.
\end{cases}
\eeq
(Note that the existence and uniqueness of the solution pair $(\sigma_h, p_h)$ follows from Section \ref{sec:approx_sols}, in view of (\ref{coercive}) and (\ref{discrete_infsup}).) Then one has the following error estimates:
\beq \label{fluid_inequ}
||\sigma - \sigma_h||_\Sigma \leq c_{1,h} \inf_{\zeta_h \in \Sigma_h} ||\sigma - \zeta_h||_\Sigma + c_{2,h} \inf_{q_h \in M_h} ||p - q_h||_M,
\eeq

\beq \label{pressure_inequ}
||p- p_h||_\Sigma \leq c_{3,h} \inf_{\zeta_h \in \Sigma_h} ||\sigma - \zeta_h||_\Sigma + c_{4,h} \inf_{q_h \in M_h} ||p - q_h||_M,
\eeq
with $c_{1,h} = \left( 1 + \frac{||a||}{\alpha_h}\right)  \left( 1 + \frac{||b||}{\beta_h}\right)$, $ c_{2,h} = \frac{||b||}{\alpha_h}$; moreover, if $M = M_h$, one can take $c_{2,h}  = 0, c_{3,h} = c_{1,h} \frac{||a||}{\beta_h}$, and $c_{4,h} = 1 + \frac{||b||}{\beta_h} + c_{2,h} \frac{||a||}{\beta_h}$. 
\end{lemma}

Concerning the efficacy of our FEM for numerically approximating the fluid-fluid structure system (\ref{varequs}) -(\ref{end_resolvent}), we have the following:

\begin{thm} \label{thm:conv_bounds}
With fluid chambers $\Om$ and $\Op$ as before, let $h > 0$ be the discretization parameter  which gives rise to the FEM subspaces $V_h^-$, $V^+_h$, $\Pi_h^-$, and $\Pi_h^+$ of (\ref{Vhspace}) - (\ref{Pi+space}), respectively. Given the solution variables $[\mathbf{u}^+, p^+, \mathbf{u}^-, p^-] \in \mathcal{D}(\mathcal{A}) \times L^2(\Om)$ of (\ref{varequs}) -(\ref{end_resolvent}), further suppose $[\mathbf{u}^+, p^+, \mathbf{u}^-, p^-] \in \mathbf{H}^2(\Op) \times H^1(\Op) \times \mathbf{H}^2(\Om) \times H^1(\Om)$ and $\mu_\l(\cdot) , \, \tilde \mu (\mathbf{f}, g) \in \mathbf{H}^2(\Op)$ and $q_\l(\cdot) , \, \tilde q (\mathbf{f}, g) \in \mathbf{H}^1(\Op)$. Then with respect to their FEM approximations $[\mathbf{u}_h^+, p_h^+, \mathbf{u}_h^-, p_h^-]$, as given by (\ref{discrete_varequs}) and (\ref{mapping_sols}), we have the following rates of convergence:
\begin{enumerate}
\item $ ||\mathbf{u}^- - \mathbf{u}^-_h||_{\mathbf{H}^1_{\partial \Om \setminus{\Gamma}}(\Om)} \leq C_\lambda \cdot h ||[\mathbf{f}, g, \mathbf{h}]||_{\mathcal{H}}$.
\item $ ||p^- - p^-_h||_{L^2(\Om)} \leq C_\lambda \cdot h ||[\mathbf{f}, g, \mathbf{h}]||_{\mathcal{H}}$.

\item 
$ ||\mathbf{u}^+ - \mathbf{u}^+_h||_{\mathbf{H}^1_{\partial \Op \setminus{\Gamma}}(\Op)} \leq C_\lambda \cdot h ||[\mathbf{f}, g, \mathbf{h}]||_{\mathcal{H}}$.

\item $ ||p^+ - p^+_h||_{L^2(\Op)} \leq C_\lambda \cdot h ||[\mathbf{f}, g, \mathbf{h}]||_{\mathcal{H}}$.

\end{enumerate}

\end{thm}

\begin{proof}
Throughout, we will appeal to Lemma \ref{lem:Guermond}, taking 
\beq 
a(\cdot, \cdot)  \equiv a_\l(\cdot, \cdot): \mathbf{H}_{\partial \Om \setminus \Gamma}^1(\Om) \times \mathbf{H}_{\partial \Om \setminus \Gamma}^1(\Om)  \to \R \nonumber
\eeq
and set
\beq 
b(\cdot, \cdot) : \mathbf{H}_{\partial \Om \setminus \Gamma}^1(\Om)  \times L^2(\Om) \to \R \text{ as in (\ref{b})} \nonumber. 
\eeq 
Recall in Section \ref{sec:Maximality}, we showed that both $a_\l(\cdot, \cdot)$ and $b(\cdot, \cdot)$ are bounded, so \\
$||a||_{\mathcal{L}([\mathbf{H}_{\partial \Om \setminus \Gamma}^1(\Om)]^2 , \R)} \leq C_\l$ and $||b||_{\mathcal{L}(\mathbf{H}_{\partial \Om \setminus \Gamma}^1(\Om) \times L^2(\Om), \R)} \leq C_\l$ for some $C_\l > 0$. Additionally, $a_\l(\cdot,  \cdot)$ was $\mathbf{H}_{\partial \Om \setminus \Gamma}^1(\Om) $-elliptic.

Thus, using (\ref{fluid_inequ}) and (\ref{um_bound}), we have
\begin{align}
||\bum - \mathbf{u}_h^-||_{\mathbf{H}^1_{\partial \Om \setminus \Gamma}(\Om)} & \leq C_{1,h} \inf_{\varphi_h \in V_h^-} || \bum - \varphi_h||_{\mathbf{H}_{\partial \Om \setminus \Gamma}^1(\Om) } + C_{2,h} \inf_{q_h \in \Pi^-_h} ||p^- - q_h||_{L^2(\Om)} \nonumber \\
& \leq \tilde C_{1,h}  h |\bum|_2 + \tilde C_{2,h} h |p^-|_1 \nonumber\\
& \leq Ch |[|\mathbf{f}, g, \mathbf{h}]||_{ \mathcal{H}}, 
\end{align}
where the last step follows from wellposedness of (\ref{compact_var_form}). 

Similar to above, using (\ref{pressure_inequ}) and (\ref{pm_bound}), we have
\begin{align}
||p^- - p^-_h||_{L^2(\Om)} & \leq C_{3,h} \inf_{\varphi_h \in V_h^-} ||\bum - \varphi_h||_{\mathbf{H}_{\partial \Om \setminus \Gamma}^1(\Om) } + C_{4,h}\inf_{q_h \in \Pi^-_h} ||p^- - q_h||_{L^2(\Om)}  \nonumber\\
& \leq C_{3,h} ||\bum - \mathbf{u}_h^- ||_{\mathbf{H}_{\partial \Om \setminus \Gamma}^1(\Om) } + C_{4,h} ||p^- - p_h^-||_{L^2(\Om)} \nonumber \\
& \leq Ch ||[\mathbf{f}, g, \mathbf{h}]||_{ \mathcal{H}},
\end{align}
where again the last step follows from wellposedness of (\ref{compact_var_form}).

For the third and fourth inequalities, we note the following result from \cite{AB} (pg. 215): Letting ${\mathbf{u}}^+_h$ be the Ritz-Galerkin  approximation and $\mathbf{u}^+_I$ be an $V_h^+$ interpolant of $\bup$, we have 
\beq \label{fluid_dis_bound}
||\bup -{\mathbf{u}}^+_h||_{\mathbf{H}^1(\Op)} \leq C ||\bup - \mathbf{u}^+_I ||_{\mathbf{H}^1(\Op)}.
\eeq
Similarly, with $p^+_h$ being the Ritz-Galerkin approximation and $p^+_I$ as a $\Pi_h^+$ interpolant of $p^+$, we have
\beq\label{pressure_dis_bound}
||p^+ - p^+_h||_{\Op} \leq C ||p^+ - p^+_I ||_\Op. 
\eeq

We will also make use of the following result (Lemma 1.130) from \cite{Guermond}:
\begin{lemma} Letting $u_I$ be the Scott-Zhang interpolant of $u$, and $k$ be the degree of polynomial basis functions, there exists some $C > 0$ such that 
\beq \label{SZestimate}
||u - u_I||_{H^m(\Omega)} \leq C h^{\ell - m} |u|_{\ell}.
\eeq
\end{lemma}

Using $\bup =\mu_\l(\bum) + \tilde \mu(\mathbf{f},g)$ and $\mathbf{u}_h^+ = \mu_{\l, h}(\mathbf{u}_h^-)  + \tilde \mu_h(\mathbf{f},g)$,  consider 
\begin{align} \label{u+discreteerror}
|| \bup - \mathbf{u}_h^+||_{\mathbf{H}^1_{\partial \Op \setminus \Gamma}(\Op)} & = ||\mu_\l(\bum) + \tilde \mu(\mathbf{f},g) - (\mu_{\l, h}(\mathbf{u}_h^-)  + \tilde \mu_h(\mathbf{f},g))||_{\mathbf{H}^1_{\partial \Op \setminus \Gamma}(\Op)} \nonumber \\
& \leq ||\mu_\l(\bum)  - \mu_{\l,h}(\bum) ||_{\mathbf{H}^1_{\partial \Op \setminus \Gamma}(\Op)} \nonumber   \\
& \hspace{.3in}+ ||\mu_{\l,h}(\bum) -(\mu_{\l, h}(\mathbf{u}_h^-) ||_{\mathbf{H}^1_{\partial \Op \setminus \Gamma}(\Op)}  \nonumber\\
& \hspace{.3in}+ ||\tilde \mu(\mathbf{f},g)  -  \tilde \mu_h(\mathbf{f},g)||_{\mathbf{H}^1_{\partial \Op \setminus \Gamma}(\Op)}. 
\end{align}

For the first term, we apply (\ref{fluid_dis_bound}), (\ref{SZestimate}), wellposedness of $\mu_\l$, and continuous dependence on data of $\bum$ to obtain:
\begin{align} \label{first_term}
||\mu_\l(\bum)  - \mu_{\l,h}(\bum) ||_{\mathbf{H}^1_{\partial \Op \setminus \Gamma}(\Op)} & \leq  ||\mu_\l(\bum)  - \mu_{\l,I}(\bum) ||_{\mathbf{H}^1_{\partial \Op \setminus \Gamma}(\Op)} \nonumber \\
& \leq Ch|\mu_\l(\bum|_{\Gamma})|_2 \nonumber\\
& \leq C_1h ||\bum||_{\mathbf{H}^{1/2}(\Gamma)} \nonumber\\
& \leq C_2 h ||\bum||_{\mathbf{H}^1(\Om)} \nonumber\\
& \leq  C_3 h || [\mathbf{f}, g, \mathbf{h}] ||_\mathcal{H}.
\end{align}

Likewise, using the wellposedness of $\tilde \mu(\mathbf{f}, g)$, for the third term we have 
\beq \label{third_term}
 ||\tilde \mu(\mathbf{f},g) -  \tilde \mu_h(\mathbf{f},g)||_{\mathbf{H}^1_{\partial \Op \setminus \Gamma}(\Op)} \leq C h ( ||\mathbf{f}||_\Op + ||g||_\Op)  \leq C h || [\mathbf{f}, g, \mathbf{h}] ||_\mathcal{H}.
 \eeq

For the second term of (\ref{u+discreteerror}), we note that $\mu_{\l, h}$ is a linear map. Thus, using the wellposedness of $\mu_{\l,h}$, it follows that 
\begin{align} \label{second_term}
||\mu_{\l,h}(\bum) -(\mu_{\l, h}(\mathbf{u}_h^-) ||_{\mathbf{H}^1_{\partial \Op \setminus \Gamma}(\Op)} &  = ||\mu_{\l,h} (\bum - \mathbf{u}^-_h)||_{\mathbf{H}^1_{\partial \Op \setminus \Gamma}(\Op)} \nonumber \\
& \leq C ||\bum - \mathbf{u}^-_h||_{\mathbf{H}^1_{\partial \Op \setminus \Gamma}(\Op)} \nonumber \\
& \leq Ch ||[\mathbf{f}, g, \mathbf{h}]||_\mathcal{H},
\end{align}
after using Part 1 of this theorem. 

Combining (\ref{first_term}), (\ref{second_term}), and (\ref{third_term}) gives

\beq 
|| \bup - \mathbf{u}_h^+||_{\mathbf{H}^1_{\partial \Op \setminus \Gamma}(\Op)} \leq Ch ||[\mathbf{f}, g, \mathbf{h}]||_\mathcal{H},
\eeq
which establishes Part 3. 

Finally, using (\ref{pressure_dis_bound}), consider 
\begin{align} \label{p+discreteerror}
||p^+ - p^+_h||_\Op  &  =  ||q_\l(\bum) + \tilde q(\mathbf{f},g) - (q_{\l,h}(\mathbf{u}_h^-) + \tilde q_h(\mathbf{f},g))||_\Op \nonumber \\
& \leq  ||q_\l(\bum)  - q_{\l,h}(\bum) ||_\Op + ||q_{\l,h}(\bum) - q_{\l,h}(\mathbf{u}_h^-)||_\Op \nonumber \\
& \hspace{.5in} + ||\tilde q(\mathbf{f},g) - \tilde q_h(\mathbf{f},g))||_\Op
\end{align}


For the first term, we apply (\ref{pressure_dis_bound}), (\ref{SZestimate}), wellposedness of $q_\l$, and continuous dependence on data of $\bum$ to obtain:
\begin{align} \label{first_pressure_term}
||q_\l(\bum)  - q_{\l,h}(\bum) ||_\Op & \leq  ||q_\l(\bum)  - q_{\l,I}(\bum) ||_{\Op} \nonumber \\
& \leq Ch^1|q_\l(\bum|_{\Gamma})|_1 \nonumber\\
& \leq C_1h ||\bum||_{\mathbf{H}^{1/2}(\Gamma)} \nonumber\\
& \leq C_2 h ||\bum||_{\mathbf{H}^1(\Om)} \nonumber\\
& \leq  C_3 h || [\mathbf{f}, g, \mathbf{h}] ||_\mathcal{H}.
\end{align}

Similarly, using the wellposedness of $\tilde q(\mathbf{f}, g)$, for the third term we have 
\beq \label{third_pressure_term}
 ||\tilde q(\mathbf{f},g) -  \tilde q_h(\mathbf{f},g)||_\Op \leq C h ( ||\mathbf{f}||_\Op + ||g||_\Op)  \leq C h || [\mathbf{f}, g, \mathbf{h}] ||_\mathcal{H}.
 \eeq

For the second term of (\ref{p+discreteerror}), we note that $q_{\l, h}$ is a linear map. Thus, using the wellposedness of $q_{\l,h}$, it follows that 
\begin{align} \label{second_pressure_term}
||q_{\l,h}(\bum) -q_{\l, h}(\mathbf{u}_h^-) ||_\Op &  = ||q_{\l,h} (\bum - \mathbf{u}^-_h)||_{\Op} \nonumber \\
& \leq C ||\bum - \mathbf{u}^-_h||_{\mathbf{H}^1_{\partial \Om \setminus \Gamma}(\Op)} \nonumber \\
& \leq Ch ||[\mathbf{f}, g, \mathbf{h}]||_\mathcal{H},
\end{align}
after using Part 1 of this theorem. 

Combining (\ref{first_pressure_term}), (\ref{second_pressure_term}), and (\ref{third_pressure_term}) gives

\beq 
|| \bup - \mathbf{u}_h^+||_{\mathbf{H}^1_{\partial \Op \setminus \Gamma}(\Op)} \leq Ch ||[\mathbf{f}, g, \mathbf{h}]||_\mathcal{H},
\eeq
which establishes Part 4.

\end{proof}

\section{Acknowledgment}

\noindent The authors would like to thank
the National Science Foundation, and acknowledge their partial funding from
NSF Grant DMS-1907823.

\bibliographystyle{abbrv}  
\bibliography{./references}

\begin{thebibliography}{10}

\bibitem{Avalos&Dvorak}
G.~Avalos and M.~Dvorak.
\newblock A new maximality argument for a coupled fluid-structure interaction,
  with implications for a divergence-free finite element method.
\newblock {\em Applicationes Mathematicae}, 35(3):259--280, 2008.

\bibitem{AGW}
G.~Avalos, P.~G. Geredeli, and J.~T. Webster.
\newblock Semigroup well-posedness of a linearized, compressible fluid with an
  elastic boundary.
\newblock {\em Discrete and Continuous Dynamical Systems - B},
  23(3):1267--1295, 2018.

\bibitem{AGW2}
G.~Avalos, P.~G. Geredeli, and J.~T. Webster.
\newblock A linearized viscous, compressible flow-plate interaction with
  non-dissipative coupling.
\newblock {\em Journal of Mathematical Analysis and Applications},
  477(1):334--356, 2019.

\bibitem{Avalos&Toundykov}
G.~Avalos and D.~Toundykov.
\newblock A uniform discrete inf-sup inequality for finite element
  hydro-elastic models.
\newblock {\em Evolution Equations and Control Theory}, 5(4):515--531, 2016.

\bibitem{AB}
O.~Axelsson and V.~A. Barker.
\newblock {\em Finite Element Solution of Boundary Value Problems, Theory and
  Computation}.
\newblock Society for Industrial and Applied Mathematics, Philadelphia, PA,
  2001.

\bibitem{BUFFA2001699}
A.~Buffa and G.~Geymonat.
\newblock On traces of functions in $W^{2,p}(\Omega)$ for Lipschitz domains in
  $\mathbb{R}^3$.
\newblock {\em Comptes Rendus de l'Acad\'emie des Sciences - Series I -
  Mathematics}, 332(8):699--704, 2001.

\bibitem{Guermond}
A.~Ern and J.-L. Guermond.
\newblock {\em Theory and Practice of Finite Elements}.
\newblock Springer-Verlag, New York, NY, 2004.

\bibitem{Galdi}
G.~P. Galdi.
\newblock An introduction to the mathematical theory of the Navier-Stokes
  equations.
\newblock {\em Springer Tracts in Natural Philosophy}, 38(3):259--280, 1994.

\bibitem{Grisvard2}
P.~Grisvard.
\newblock {\em Singularities in Boundary Value Problems}.
\newblock Springer-Verlag, Berlin, 1992.

\bibitem{Kato}
T.~Kato.
\newblock {\em Perturbation Theory for Linear Operators}.
\newblock Springer-Verlag, New York, NY, 1980.

\bibitem{Kesavan}
S.~Kesavan.
\newblock {\em Topics in Functional Analysis and Applications}.
\newblock New Age International Limited, Publishers, New Delhi, 1989.

\bibitem{Lasiecka&Triggiani}
I.~Lasiecka and R.~Triggiani.
\newblock {\em Control Theory for Partial Differential Equations, Continuous
  and Approximation Theories, Vols. I and II}.
\newblock Cambridge University Press, New York, NY, 2000.

\bibitem{Lemarie}
F.~Lemari\'e, E.~Blayo, and L.~Debreu.
\newblock Analysis of ocean-atmosphere coupling algorithms: consistency and
  stability.
\newblock {\em Procedia Computer Science}, 51:2066--2075, 2015.

\bibitem{Lions&Magenes}
J.~L. Lions and E.~Magenes.
\newblock {\em Non-Homogeneous Boundary Value Problems and Applications}.
\newblock Springer-Verlag, New York, 1972.

\bibitem{mclean2000strongly}
W.~McLean and W.~C.~H. McLean.
\newblock {\em Strongly elliptic systems and boundary integral equations}.
\newblock Cambridge University Press, 2000.

\bibitem{necas2011direct}
J.~Necas.
\newblock {\em Direct methods in the theory of elliptic equations}.
\newblock Springer Science \& Business Media, 2011.

\bibitem{Pazy}
A.~Pazy.
\newblock {\em Semigroups of Linear Operators and Applications to Partial
  Differential Equations}.
\newblock Springer-Verlag, New York, NY, 1983.

\bibitem{Skamarock}
W.~C. Skamarock, J.~B. Klemp, and et. al.
\newblock A description of the advanced research wrf model version 2.
\newblock Technical Report NCAR/TN?468+STR, National Center for Atmospheric
  Research, Mesoscale and Microscale Meteorology Division, Boulder, CO, June
  2005.

\end{thebibliography}

\end{document}